\documentclass[11pt,a4paper]{article}

\usepackage{psfrag}
\usepackage{amsmath,amssymb,cite}
\usepackage{latexsym}
\usepackage{graphicx}
\usepackage{lscape}
\usepackage{afterpage}

\usepackage{hyperref}
\hypersetup{
    colorlinks=true,
    linkcolor=blue,
    filecolor=magenta,      
    urlcolor=cyan,
    citecolor=red,
}
\DeclareMathOperator*{\argmin}{argmin}
\DeclareMathOperator*{\sgn}{sgn}

\newcommand{\thetab}{\overline{\theta}}
\newcommand{\ds}{\displaystyle}
\newcommand{\nexto}{\kern -0.54em}
\newcommand{\dR}{{\rm {I\ \nexto R}}}

\newcommand{\dZ}{{\cal Z \kern -0.7em Z}}
\newcommand{\dC}{{\rm\hbox{C \kern-0.8em\raise0.2ex\hbox{\vrule
height5.4pt width0.7pt}}}}
\newcommand{\dQ}{{\rm\hbox{Q \kern-0.85em\raise0.25ex\hbox{\vrule
height5.4pt width0.7pt}}}}
\newcommand{\proofbox}{\hspace{\fill}{$\Box$}}
\newtheorem{lemma}{Lemma}
\newtheorem{theorem}{Theorem}

\newtheorem{proposition}{Proposition}

\newtheorem{remark}{Remark}

\newenvironment{proof}{Proof.}{\proofbox}

\begin{document}

\author{Authors}

\author{
C. Yal{\c c}{\i}n Kaya\footnote{School of Information Technology and Mathematical 
Sciences, University of South Australia, Mawson Lakes, S.A. 5095, Australia. E-mail: yalcin.kaya@unisa.edu.au\,.\ \ ORCID: https://orcid.org/0000-0001-7962-7153}
}

\title{\vspace{0mm}\bf Markov--Dubins Interpolating Curves}

\maketitle

\begin{abstract} 
{\noindent\sf  A realistic generalization of the Markov--Dubins problem, which is concerned with finding the shortest planar curve of constrained curvature joining two points with prescribed tangents, is the requirement that the curve passes through a number of prescribed intermediate points/nodes.  We refer to this generalization as the Markov--Dubins interpolation problem.  We formulate this interpolation problem as an optimal control problem and obtain results about the structure of its solution using optimal control theory.  The Markov--Dubins interpolants consist of a concatenation of circular ($C$) and straight-line ($S$) segments.  Abnormal interpolating curves are shown to exist and characterized; however, if the interpolating curve contains a straight-line segment then it cannot be abnormal.  We derive results about the stationarity, or criticality, of the feasible solutions of certain structure.  In particular, any feasible interpolant with arc types of $CSC$ in each stage is proved to be stationary, i.e., critical.   We propose a numerical method for computing Markov--Dubins interpolating paths.  We illustrate the theory and the numerical approach by four qualitatively different examples.}
\end{abstract}

\begin{verse} {\em Key words}\/: {\sf Markov--Dubins path, Interpolation, Constrained curvature, Optimal control, Singular control, Bang--bang control, Abnormal optimal control problem.}
\end{verse}

\begin{verse} 
{\bf AMS subject classifications.} {\sf Primary 49J15, 49K15\ \ Secondary 65K10, 90C30}
\end{verse}

\pagestyle{myheadings}
\thispagestyle{plain}
\markboth{\sf\scriptsize C. Y. Kaya}{\sf\scriptsize Markov--Dubins Interpolating Curves\ \ by\ C. Y. Kaya}

\section{Introduction}

We define a {\em Markov--Dubins interpolating curve} as the shortest ${\cal C}^1$ and piecewise-${\cal C}^2$ planar curve\linebreak $z:[0,t_N]\longrightarrow\dR^2$ that passes through $(N+1)$ points, $p_0,p_1,\ldots,p_{N-1},p_N$, $N\ge 1$, prescribed at $0$ and at the {\em free} parameter values $0<t_1<\ldots<t_{N-1}<t_N$, where the slopes, i.e., the velocities, at $p_0$ and $p_N$ are also prescribed, such that the curvature of the path $z(t)$ at almost every point is not greater than $a>0$.  Note that the parameters $t_1,t_2,\ldots,t_N$ are unknown; so, they also are to be determined.  The problem of finding a Markov--Dubins interpolating curve can then be posed as follows.
\[
\mbox{(P)}\left\{\begin{array}{rl}
\min &\ t_N \\[2mm]
\mbox{s.t.} &\ z(t_0) = p_0\,,\ z(t_1) = p_1\,,\ldots,\ z(t_N) = p_N\,,\\[2mm]
  &\ \dot{z}(t_0) = v_0,\ \dot{z}(t_N) = v_N\,,
                   \\[2mm]
  &\ \|\ddot{z}(t)\|\le a\,,\ \|\dot{z}(t)\| = 1\,,\mbox{ for a.e. }  t\in[0,t_N]\,,
\end{array}\right.
\]
where $\dot{z} = dz/dt$, $\ddot{z} = d^2z/dt^2$, and $\|\cdot\|$ is the Euclidean norm.  By continuity, we clearly have $\|v_0\| = \|v_N\| = 1$.  We further make the obvious assumption that $p_{i-1}\neq p_i$, $i=1,\ldots,N$.

\subsection{Existing results on the Markov--Dubins problem}
The special case of Problem~(P) with $N=1$ is the celebrated {\em Markov--Dubins problem}, a solution curve of which is referred to as {\em Markov--Dubins path}, which turns out to be a concatenation of circular subarcs and a straight line, as proved by Lester Eli Dubins~\cite{Dubins1957} in 1957, although the problem was first posed and some instances studied by Andrey Andreyevich Markov~\cite{Markov1889, KreNud1977} in 1889---which explains the term ``Markov--Dubins path."  Suppose that a circular arc is represented by $C$ and a straight line segment by $S$.  Dubins' elegant result asserts that the sequence of concatenated arcs in such a shortest path can be of type $CSC$, $CCC$, or a subset thereof. 

We have recently studied an optimal control formulation of the Markov--Dubins problem and reproduced Dubins' result using optimal control theory and perturbation techniques~\cite[Theorem~1]{Kaya2017}, as was also done with slightly different approaches, in~\cite{BoiCerLeb1991, SusTan1991}.  The study in \cite{Kaya2017} has presented the following additional contributions.
\begin{itemize}
\item Abnormal optimal control solutions (when the optimal multiplier of the objective functional is zero) do exist and are characterized as curves of either type $CC$ or type $C$~\cite[Lemmas~5 and 7 and Corollaries~1 and 2]{Kaya2017}.  
\item Any feasible path of the types listed in Dubins' 1957 result is a stationary solution, i.e., that these feasible paths satisfy the maximum principle (or the necessary conditions of optimality)~\cite[Theorem~2]{Kaya2017}.  
\item Exploiting the structure of the optimal solution and using arc parameterization techniques~\cite{KayNoa1996, KayNoa2003, MauBueKimKay2005}, a numerical method has been proposed and illustrated via examples~\cite[Section~5]{Kaya2017}, including the abnormal case. 
\end{itemize}
For a survey of the other studies related to the Markov--Dubins problem, see \cite{Kaya2017}.

\subsection{Interpolation and the Markov--Dubins problem}
Problem~(P) is a generalization of the Markov--Dubins problem in the sense that the curvature-constrained curve between two given oriented points is required to pass through a number of prescribed intermediate points. Many of the results obtained in \cite{Kaya2017} serve as building blocks for the results for the interpolation problem~(P), in the present paper.

Reformulation of interpolation problems as optimal control problems is not new.  For example, Kaya and Noakes study in \cite{KayNoa2013} interpolating curves in $\dR^n$, for any $n\ge1$, minimizing the $L^\infty$-norm of the acceleration vector, using optimal control theory.  Their reformulation gives rise to an optimal control problem with intermediate constraints or a multi-stage optimal control problem, which can be effectively treated by using the optimal control theory and implementation in \cite{AugMau2000, ClaVin1989, DmiKag2011}.  The interpolation problem in \cite{KayNoa2013} is markedly different from the one studied in the present paper, however the optimal control approach we will adopt is similar.

We now describe some earlier work also using optimal control in the study of interpolating curves:  The 1975 work of McClure~\cite{McClure1975} formulates an optimal control problem to answer the question of existence of a perfect spline.  Aronsson extends in~\cite{Aronsson1979} McClure's work to more general objective functionals; however, both of these early works are concerned with scalar interpolating functions, i.e., the space they work in is $\dR$, rather than $\dR^2$.  Fredenhagen, Oberle and Opfer \cite{FreObeOpf1999, OpfObe1988} study restricted as well as monotone cubic spline interpolants by treating the interpolation problem as an optimal control problem. Agwu and Martin's study in \cite{AgwMar1998} is along similar lines. Interpolating curves minimizing various other criteria (but not the criterion and the setting we consider in this paper) are studied by means of optimal control by Isaev in~\cite{Isaev2010}.

Although Markov--Dubins path has been extensively studied both theoretically and practically, finding as wide range of applications as the path planning of drones (or uninhabited aerial vehicles) and robots, and the tunnelling in underground mines, to the author's knowledge, its generalization to interpolation has not been studied in its entirety or true form yet. Brunnett, Kiefer and Wendt~\cite{BruKieWen1999} consider Problem~(P) (but not in the form we pose it here) with a bound on the average, rather than the pointwise, curvature, and propose an algorithm for finding what they call a ``Dubins spline" with the following ad hoc steps: (Step 1)~Estimate/guess the missing velocity directions at the interior nodes, (Step 2)~Find a Markov--Dubins path of type $CSC$ (but not of type $CCC$ or a subset thereof) between each two consecutive nodes and (Step 3)~Update the estimated missing velocity directions at the interior nodes using some nonlinear optimization procedure to minimize the overall length of the spline; [if some stopping criterion is not satisfied, then] go to Step 2.  In~\cite{BruKieWen1999}, not only the problem that is solved is different from Problem~(P), but also the theory and pertaining analysis are not adequately covered.  Similar ad hoc approaches ultimately leading to interpolants which are suboptimal solutions, of problems related to Problem~(P), can be found in the literature; see, for example, \cite{SavFraBul2008, IsaShi2015, Looker2011}.

Relatively more recently, Goaoc, Kim and Lazard \cite{GoaKimLaz2013} studied Problem~(P), by using the same assumption that was also used in \cite{BruKieWen1999}:  the optimal path between any two consecutive nodes will be of type $CSC$.  To guarantee this, every two consecutive nodes are assumed to be placed farther than $4/a$ units apart.  Via this assumption, they reduce Problem~(P) to a problem where one needs to find the $(N-1)$ missing velocity directions, or angles, at the interior nodes.  They define a function of these unknown angles and state that this function is locally strictly convex over a certain domain.  They present a result concluding that a solution of Problem~(P), under the stated assumptions, can be found by solving up to $2^{N-3}$ convex problems with up to $4(N-1)$ inequality constraints in each convex problem, presumably the exponential number reflecting the combinatorial nature of the problem.  The work in~\cite{GoaKimLaz2013} does not provide a mathematical description of the convex problems in standard form.  It does not provide an algorithm or numerical experiments, either.  So, one cannot implement the approach in \cite{GoaKimLaz2013} and make comparisons to test its efficacy.  We keep in mind that~\cite{GoaKimLaz2013} is concerned only with interpolants whose curve segments between two consecutive nodes are of type $CSC$.  We note again that in this particular approach the effort to find a solution grows exponentially with the number of nodes, as one would expect.

Given the above background concerning interpolation and the Markov--Dubins problem, there is an obvious need to perform a full analysis of Problem~(P) and develop new numerical methods based on this analysis to solve Problem~(P).  It is also of our concern that, in the existing papers, which present ad hoc numerical techniques, almost no numerical experiment can be reproduced, because of either a lack of complete description of an algorithm or the incompleteness of the data used to conduct the experiment.

Markov-Dubins interpolating curves would be applicable, for example, to land and marine surveillance, including military and civilian search-and-rescue operations, and agriculture, where it might be desirable to obtain images of a complicated terrain (land or sea bed) at a sequence of specified locations (nodes, in this case) by means of an aerial vehicle or a sea vessel which must follow the shortest route through these nodes.  The current paper is also driven by the need of reliable computational techniques to find optimal paths for such practical applications.

\subsection{Contributions of the current paper}
The approach in the present paper uses optimal control theory to derive the necessary conditions of optimality, after reformulating the interpolation problem~(P), first as a multi-point constrained optimal control problem~(Pc), and then as a multistage, or multiprocess, optimal control problem~(Pmc).  This approach is similar to that adopted in \cite{KayNoa2013}, but it is implemented here for a different problem.  

Since the optimality conditions furnish conditions for optimal curves between each two consecutive nodes, as well as some additional conditions at the nodes, Lemmas~1--7 in this paper turn out to be companions of those obtained for the single-stage problem in~\cite{Kaya2017}.  The proofs of Lemmas~1--7 (except some parts as indicated) are obtained along similar lines to those in~\cite{Kaya2017}; so, the proofs of Lemmas~1--7 in this paper are given briefly, also indicating the similarities and dissimilarities to those in~\cite{Kaya2017}.  It should be noted that, in addition to optimal control theory, which alone is not enough, perturbation analysis is also utilized in Lemmas~\ref{abnormal_int} and \ref{CCCC_int}.

The first main result of the present paper is given in Theorem~\ref{Dubins_int}.  It extends Dubins' theorem from finding the shortest curvature-constrained curve between two oriented points to finding the shortest curvature-constrained interpolating curve passing through intermediate nodes between two oriented end points.  The structure of the solution interpolating curve in each stage, i.e., the piece between each two consecutive nodes, is again one of either type $CCC$ (bang--bang--bang) or $CSC$ (bang--singular--bang) or a subset thereof, with additional conditions imposed at the nodes.  As in~\cite{Kaya2017}, abnormal interpolating curves are shown to exist.  However, by Remark~\ref{rem:abnormal}, if there is even a single straight-line segment in the whole interpolating curve, then the solution cannot be abnormal.

Proposition~\ref{subarcs_int} states that under the reasonable assumption of continuous curvature at the (interior) nodes, the total number of subarcs in the Markov--Dubins interpolating curve is at most $(2N+1)$, which is considerably smaller than $3N$ for large $N$.

Theorems~\ref{stationarity_int} and \ref{stationarity_int_abnormal} provide conditions, under which, any feasible solution of Problem~(P) of the possible types listed in Theorem~\ref{Dubins_int} is stationary.  Theorem~\ref{stationary_CSC} states a stronger/particular result in that any feasible solution with arc types of $CSC$ in each stage is a stationary solution.  

Proposition~\ref{lengths_CSC}(b) states that, given two consecutive stages of type $CSC$, the node is placed at the ÒmidpointÓ of the common $C$-subarc.  This is a computationally useful feature, as it can be employed to facilitate/speed up convergence.  We note that a proof of this result is given in~\cite[Lemma~4.1]{GoaKimLaz2013}.  In the current paper, we provide a proof which is different from/alternative to, and arguably shorter than, that in~\cite{GoaKimLaz2013}.

Since the solution structure of Problem~(P) is shown to be a concatenation of bang and singular arcs, we parameterize the problem with respect to the subarcs for the whole interpolating curve that is sought after.  Hence, we transform Problem~(P) into the finite-dimensional optimization problem~(Ps) in terms of just the subarc lengths, and do it in a rather neat form.  The {arc parameterization technique} that is implemented here to obtain Problem~(Ps) comes from~\cite{KayNoa1994, KayNoa1996, KayLucSim2004, KayNoa2003,  MauBueKimKay2005}, another implementation of which can also be found \cite{KayMau2014}.


Four qualitatively different examples involving Markov--Dubins interpolating curves are studied.  The numerical details are given in sufficient detail so that they can be cross-checked/verified, bearing also in mind that they may constitute test-bed examples in related future studies.  It should be noted that a solution to Problem~(Ps) can be found with a high precision even in the cases when getting a solution to Problem~(Pc) is not possible via direct discretization.

The paper is organized as follows. In Section~\ref{optcont}, we transform Problem~(P) first into a time-optimal control problem and then a multi-stage optimal control problem.  We obtain the necessary conditions of optimality.  In Section~\ref{Markov-Dubins_int}, we provide the preliminary result to lay the ground and prove the first main result of the paper, Theorem~\ref{Dubins_int}.  Proposition~\ref{subarcs_int}, a result on the total number of subarcs is presented in Section~\ref{sec:stationarity}.  Section~\ref{sec:stationarity} also discusses stationarity of a class of feasible solutions stated in Theorems~\ref{stationarity_int}--\ref{stationary_CSC}.  Section~\ref{themethod} describes the numerical approach, while Section~\ref{experiments} presents the numerical experiments.  Finally, Section~\ref{conclusion} concludes the paper, with a discussion and a short list of open problems.

\newpage

\section{Optimal Control Formulation and a Maximum Principle}
\label{optcont}

Just like the Markov--Dubins problem studied in~\cite{Kaya2017},  Problem~(P) can be equivalently cast as an optimal control problem, albeit in a more general form, as follows.  Let $z(t) := (x(t), y(t))\in\dR^2$, with $\dot{x}(t) := \cos\theta(t)$ and $\dot{y}(t) := \sin\theta(t)$, where $\theta(t)$ is the angle the velocity vector $\dot{z}(t)$ of the curve $z(t)$ makes with the horizontal.  These definitions readily verify that $\|\dot{z}(t)\| = 1$.  One also has that $p_0 = (x_0,y_0)$ and $p_N = (x_f,y_f)$.  Moreover, $\|\ddot{z}\|^2 = \ddot{x}^2+\ddot{y}^2 = \dot{\theta}^2$.  Therefore, $|\dot{\theta}(t)|$ is nothing but the {\em curvature}.  The quantity $\dot{\theta}(t)$, on the other hand, which can be positive or negative, is referred to as the {\em signed curvature}.  Consider, figuratively, a vehicle travelling along a circular path.  If $\dot{\theta}(t) > 0$ then the vehicle travels in the counter-clockwise direction, i.e., it {\em turns left} (along an {\em $L$-subarc}), and if $\dot{\theta}(t) < 0$ then the vehicle travels in the clockwise direction, i.e., it {\em turns right} (along an {\em $R$-subarc}).

Let $u(t) := \dot\theta(t)$.  Suppose that the {\em angles of the slopes} (of directions) with the horizontal, at the points $p_0$ and $p_N$ are denoted by $\theta_0$ and $\theta_f$, respectively.  Problem~(P) can then be re-written as a {\em time-optimal} (or {\em minimum-time}) {\em control problem}, where $x$, $y$ and $\theta$ are the {\em state variables} and $u$ the {\em control variable}:

\[
\mbox{(Pc)}\left\{\begin{array}{rll}
\min &\ t_N & \\[2mm]
\mbox{s.t.} &\ \dot{x}(t) = \cos\theta(t)\,, & x(0) = x_0\,,\
         x(t_1) = x_1\,,\ldots,\ x(t_N) = x_f\,, \\[2mm] 
  &\ \dot{y}(t) = \sin\theta(t)\,, & y(0) = y_0\,,\ y(t_1) =
         y_1\,,\ldots,\ y(t_N) = y_f\,,\\[2mm] 
  &\ \dot{\theta}(t) = u(t)\,, & \theta(0) = \theta_0\,,\ \theta(t_N) =
    \theta_f\,,\\[3mm]
  & & |u(t)|\le a\,,\mbox{ for a.e. }  t\in[0,t_N]\,.
\end{array}\right.
\]
The marked difference between Problem~(Pc) and the standard optimal control problem representing the Markov--Dubins problem in~\cite{Kaya2017} (for $N=1$) is that in Problem~(Pc) the state variables $x(t)$ and $y(t)$ are specified, i.e., they are constrained to take certain values, at the intermediate unknown time points $t_1,t_2,\ldots,t_{N-1}$.  These constraints are often referred to in the optimal control literature as {\em interior point state constraints}.   Problem~(Pc) can be further transformed into a {\em multiprocess}, or {\em multistage}, {\em optimal control problem}, as in \cite{KayNoa2013}, where a different class of interpolating curves, namely, the class of interpolating curves minimizing their pointwise maximum acceleration, was studied. 

A maximum principle, i.e., necessary conditions of optimality, for multistage problems is provided by Clarke and Vinter in \cite{ClaVin1989} for rather general problems, including problems which are not differentiable, for which the transversality conditions are presented by means of generalized derivatives and normal cones. Augustin and Maurer \cite{AugMau2000} transform the multistage control problem for a special class of systems (including the class we have in this paper) into a single-stage one by means of a standard rescaling of the unknown stage time durations to unity (defined below). This allows the transversality conditions to be described more simply. Dmitruk and Kaganovich~\cite{DmiKag2011} study optimal control problems with intermediate state constraints.  We will make use of the references \cite{AugMau2000,ClaVin1989,DmiKag2011}, as well as \cite{KayNoa2013}, where in the latter reference a similar setting was employed for an entirely different interpolation problem, in writing the necessary conditions of optimality.

Define a new time variable $s$ in terms of $t$ to map the arc length $\tau_i$ of each {\em stage} $i$ to unity as follows.
\[
t = t_{i-1} + s\,\tau_i\,,\quad s\in[0,1]\,,\quad \tau_i := t_i - t_{i-1}\,,\quad i = 1,\ldots,N\,.
\]
With this definition, the time horizon of each stage $i$ is rescaled as $[0,1]$ in the new independent (time) variable $s$.  Note that, since $p_{i-1}\neq p_i$, we have $\tau_i > 0$, $i=1,\ldots,N$.  Let
\[
x^i(s) := x(t)\,,\ y^i(s) := y(t)\,,\ \theta^i(s) := \theta(t)\,,\ \mbox{and}\ u^i(s) := u(t)\,,\ \mbox{for } s\in[0,1],\ t\in[t_{i-1},t_{i}]\,,
\]
for $i = 1,\ldots,N$.  Here $x^i$ denotes the values of the state variable $x$ in stage $i$, and other stage variables are to be interpreted similarly.  The role of the superscript $i$ should be understood as an {\em index} rather than a {\em power}---this should become clear from the context. With the usage of stages one needs to pose constraints to ensure continuity of the state variables at the junction of any two consecutive stages:
\[
x^i(1) = x^{i+1}(0)\,,\quad y^i(1) = y^{i+1}(0)\,,\quad \theta^i(1) = \theta^{i+1}(0)\,,
\]
for $i = 1,\ldots,N-1$.  The resulting single-stage optimal control problem can now be written as
\[
\mbox{(Pmc)}\left\{\begin{array}{rll}
\min &\ \ds \sum_{i=1}^N\,\tau_i = \sum_{i=1}^N\,\int_0^1 \tau_i\,ds  & \\[5mm]
\mbox{s.t.} &\ \dot{x}^i(s) = \tau_i\,\cos\theta^i(s)\,, & x^i(0) = x_{i-1}\,,\ x^i(1) = x_i\,, \\[2mm] 
&\ \dot{y}^i(s) = \tau_i\,\sin\theta^i(s)\,, & y^i(0) = y_{i-1}\,,\ y^i(1) = y_i\,, \\[2mm]  
&\ \dot{\theta}^i(s) = \tau_i\,u^i(s)\,, & \theta^1(0) = \theta_0\,,\ \theta^N(1) = \theta_f\,, \\[2mm]
  & & |u^i(s)|\le a\,,\ \ \ i = 1,\ldots,N\,,\\[2mm]
  & & x^{j+1}(0) = x^j(1)\,,\ y^{j+1}(0) = y^j(1)\,, \\[2mm]
  & & \theta^{j+1}(0) = \theta^j(1)\,,\ \ \ j = 1,\ldots,N-1\,.
\end{array}\right.
\]
In what follows, we will state a maximum principle, i.e., necessary conditions of optimality, for Problem~(Pmc), using \cite[Theorem~3.1 and Corollary~3.1]{ClaVin1989} and \cite[Section~4]{AugMau2000} or~\cite{DmiKag2011}.  First, define the Hamiltonian function for the $i$th stage of Problem~(Pmc) as
\[
H^i(x^i,y^i,\theta^i, \lambda_0,\lambda^i_1,\lambda^i_2,\lambda^i_3,u^i) := \tau_i\left(\lambda_0
+ \lambda^i_1\,\cos\theta^i + \lambda^i_2\,\sin\theta^i + \lambda^i_3\,u^i\right),
\]
where $\lambda_0$ is a scalar (multiplier) parameter, and $\lambda_j^i:[0,1]\rightarrow\dR$, $j=1,2,3$, are the adjoint variables (or multiplier functions) in the $i$th stage.  Let
\[
H^i[s] :=
H^i(x^i(s),y^i(s),\theta^i(s), \lambda_0,\lambda^i_1(s),\lambda^i_2(s),\lambda^i_3(s),u^i(s))\,.
\]
Suppose that $x^i,y^i,\theta^i\in W^{1,\infty}(0,1;\dR)$, $u^i\in L^\infty(0,1;\dR)$, and $\tau_i\in[0,M)$, $i = 1,\ldots,N$, where $M$ is large enough so that $\max_i\tau_i<M-\varepsilon$ with $\varepsilon>0$, solve Problem~(Pmc).  Then there exist a number $\lambda_0\ge0$ and functions $\lambda_j^i\in W^{1,\infty}(0,1;\dR)$, $j=1,2,3$, such that $\lambda^i(s):= (\lambda_0,\lambda_1^i(s), \lambda_2^i(s), \lambda_3^i(s)) \neq \bf0$, for every $s\in[0,1]$, $i = 1,\ldots,N$, and, in addition to the state differential equations and other constraints given in Problem~(Pmc), the following conditions hold:
\begin{eqnarray}
&& \dot{\lambda}_1^i(s) = -H^i_{x}[s]\,,\ \dot{\lambda}_2^i(s) = -H^i_{y}[s]\,,\ \dot{\lambda}_3^i(s) = -H^i_{\theta}[s]\,,\mbox{ a.e. } s\in[0,1],\ i = 1,\ldots,N, \label{lambdax}  \\[1mm]
&& \lambda_j^{i+1}(0) =  \lambda_j^i(1) + \delta_j^i\,,\ \ j = 1,2,\ i = 1,\ldots,N-1, \label{l_jump}   \\[1mm]
&& \lambda_3^{i+1}(0) =  \lambda_3^i(1)\,,\ \ i = 1,\ldots,N-1, \label{lamb3_cont}   \\[1mm]
&& u^i(s)\in\argmin_{|v|\le a} H^i(x^i(s),y^i(s),\theta^i(s),\lambda_0,\lambda_1^i(s),\lambda_2^i(s),\lambda_3^i(s),v)\,,\ \mbox{a.e. } s\in[0,1]\,,  \label{controli} \\[1mm]
  && H^i[s] = 0\,,\ \ \mbox{for all } s\in[0,1]\,,\ \ i = 1,\ldots,N\,, \label{Hi_zero}
\end{eqnarray}
where $\delta_j^i$, $j = 1,2$, $i = 1,\ldots,N-1$, are real constants.  

Conditions~\eqref{lambdax}--\eqref{l_jump} state that the adjoint variables $\lambda_1^i(s)$ and $\lambda_2^i(s)$ are constant but  might have jumps as they go from one stage to the other.  On the other hand, the transversality condition~\eqref{lamb3_cont} asserts that  $\lambda_3$ is continuous at the junctions/nodes.

We define the "overall" adjoint variables $\lambda_j(t)$, $j=1,2,3$, formed by concatenating the stage adjoint variables, as follows.
\[
\lambda_j(t) := \lambda_j^i(s)\,,\quad  t = t_{i-1} + s\,\tau_i,\quad
s\in[0,1]\,,\quad \tau_i := t_i - t_{i-1}\,,\quad i = 1,\ldots,N\,.
\]
The optimality conditions \eqref{lambdax}--\eqref{Hi_zero} can now be re-written more explicitly, along with the state equations, as follows.
\begin{eqnarray}
 \dot{x}(t) &=& \cos\theta(t)\,,\quad x(0) = x_0\,,\ x(t_1) = x_1\,,\ldots,\ x(t_N) = x_f\,,\mbox{ for all } t\in[0,t_N]\,, \label{x_eqn} \\[1mm] 
 \dot{y}(t) &=& \sin\theta(t)\,,\quad y(0) = y_0\,,\ y(t_1) = y_1\,,\ldots,\ y(t_N) = y_f\,,\mbox{ for all } t\in[0,t_N]\,, \label{y-eqn} \\[1mm] 
 \dot{\theta}(t) &=& u(t)\,,\quad \theta(0) = \theta_0\,,\ \theta(t_N) = \theta_f\,,\mbox{ a.e. } t\in[0,t_N]\,, \label{theta-eqn} \\[1mm]
 \lambda_j(t) &=& \overline{\lambda}_j^i\,,\mbox{ for all } t\in[t_{i-1},t_i)\,,\ j = 1,2\,,\ i = 1,\ldots,N\,,\label{adjoint12}  \\[1mm]
 \dot{\lambda}_3(t) &=& \overline{\lambda}_1^i \,\sin\theta^i(t) - \overline{\lambda}_2^i \,\cos\theta^i(t)\,,\mbox{ for all } t\in[t_{i-1},t_i]\,,\ i = 1,\ldots,N\,,  \label{adjoint3_DE} \\[1mm]
\lambda_3(t_i^+) &=&  \lambda_3(t_i^-)\,,\ \ i = 1,\ldots,N-1, \label{lamb3_cont2}   \\[1mm]
u(t) &=& \left\{\begin{array}{ll}
\ \ a\,, & \mbox{if}\  \lambda_3(t) < 0\,, \\[3mm]
-a\,, & \mbox{if}\ \lambda_3(t) > 0\,, \\[3mm]
\mbox{undetermined}\,, & \mbox{if}\ \lambda_3(t) = 0\,,\mbox{ a.e. } t\in[t_{i-1},t_i)\,,\ i = 1,\ldots,N\,,
\end{array}\right. \label{uint} \\[1mm]
 0 &=& \lambda_0 + \overline{\lambda}^i_1\,\cos\theta(t) + \overline{\lambda}^i_2\,\sin\theta(t) + \lambda_3(t)\,u(t),\mbox{ for all } t\in[t_{i-1},t_i),\ i = 1,\ldots,N,  \label{Hizero}
\end{eqnarray}
where $\overline{\lambda}_j^i$, $i = 1,\ldots,N$, $j = 1,2$, are real constants, and we have also used the fact that $\tau_i > 0$ and the continuity condition in~\eqref{lamb3_cont}.  In the continuity condition~\eqref{lamb3_cont2}, $\lambda_3(t_i^+) := \lim_{t\to t_i^+}\lambda_3(t)$ and $\lambda_3(t_i^-) := \lim_{t\to t_i^-}\lambda_3(t)$.  Define the new constants $\rho_i$ and $\phi_i$ as
\[
\rho_i := \sqrt{\left(\overline{\lambda}_1^i\right)^2 + \left(\overline{\lambda}_2^i\right)^2}\,,\qquad 
\tan\phi_i := \frac{\overline{\lambda}_2^i}{\overline{\lambda}_1^i}\,,\ \ i = 1,\ldots,N\,.
\]
Then Equation \eqref{adjoint3_DE} can be re-written, for all $t\in[t_{i-1},t_i]$\,, $i = 1,\ldots,N$, as
\begin{equation}  \label{adjoint3_DE2}
\dot{\lambda}_3(t) = \rho_i\,\sin(\theta(t) - \phi_i)\,,
\end{equation}
and \eqref{Hizero} as
\begin{equation}  \label{Hzero2}
\lambda_3(t)\,u(t) + \rho_i\,\cos(\theta(t) - \phi_i) + \lambda_0 = 0\,.
\end{equation}
It should be noted that the adjoint variable $\lambda_3$ is nothing but the {\em switching function} for the optimal control $u$.

\section{Markov--Dubins Interpolating Curves}
\label{Markov-Dubins_int}

Observe that for the classical Markov--Dubins problem, for which $N=1$, the necessary conditions of optimality~\eqref{x_eqn}--\eqref{Hzero2}, excluding the continuity condition \eqref{lamb3_cont2}, are identical to those given in~\cite{Kaya2017}.  In the $i$th stage of the interpolation problem, i.e., when $t_{i-1} < t < t_i$, $i = 1,\ldots,N$, the necessary conditions of optimality~\eqref{x_eqn}--\eqref{Hizero}, except \eqref{lamb3_cont2} and the free-end conditions for one or both of $\theta(t_{i-1})$ and $\theta(t_i)$, are the same as those of the Markov--Dubins problem.  In Lemmas~\ref{singular_int}--\ref{straight_int}, \ref{abnormal_int}(a) and \ref{CCCC_int} for stage~$i$, $i = 1,\ldots,N$, that we present in this section, the proofs are similar to those of Lemmas~1--7 in~\cite{Kaya2017}.  Therefore, we provide the proofs of the new lemmas here as a broad summary of the related proofs in~\cite{Kaya2017}, but we refer to the particular conditions and definitions we give in~\eqref{x_eqn}--\eqref{Hzero2}, and of course use the terminology of the interpolation problem we study in this paper.  Lemma~\ref{abnormal_int}(b) and Theorem~\ref{Dubins_int} deal with the more general case of Markov--Dubins interpolating curves.

The lemma presented below collects together the companions of Lemmas~1 and 2 in~\cite{Kaya2017}.
\newpage
\begin{lemma}[Singular Interpolant Segments]  \label{singular_int}
Suppose that optimal control $u(t)$ for Problem~{\em (Pc)} is singular over some interval $[\zeta_1,\zeta_2)\subset[t_{i-1},t_i)$.  Then         \\[-8mm]
\begin{itemize}
\item[(a)] $\rho_i = \lambda_0 > 0$, i.e., the problem is normal.
\item[(b)] $\theta(t)$ is constant, i.e., $u(t) = 0$, for all $t\in[\zeta_1,\zeta_2)$.
\end{itemize}
\end{lemma}
\begin{proof}
The proof of part~(a) is furnished after equating the right-hand side of \eqref{adjoint3_DE2} to zero and considering each of the subsequent cases, similarly as in the proof of Lemma~1 in~\cite{Kaya2017}. The proof of part~(b), on the other hand, is obtained by making use of \eqref{adjoint3_DE2} and \eqref{Hzero2}, along lines similar to those in the proof of Lemma~2 in~\cite{Kaya2017}.
\end{proof}

\begin{remark} \label{rem:optcontr} \rm
From \eqref{uint} and Lemma~\ref{singular_int}(b), the optimal control can simply be written as $u(t) = -a\,\sgn(\lambda_3(t))$, a.e. $t\in[0,t_N]$.
\endproof
\end{remark}

The following lemma is a companion of Lemma~3 in~\cite{Kaya2017}.

\begin{lemma} [Differential Equation in \boldmath$\lambda_3$] \label{lambda3_lemma_int}
The adjoint variable $\lambda_3$ for Problem~{\em (Pc)} solves the differential equation
\begin{equation} \label{lambda3_DE_int} 
\dot{\lambda}_3^2(t) + \left(a\,|\lambda_3(t)| - \lambda_0\right)^2 = \rho_i^2\,,\mbox{ for all } t\in[t_{i-1},t_i)\,,\ i = 1,\ldots,N\,. 
\end{equation}
\end{lemma}
\begin{proof}
The proof is obtained, after squaring both sides of \eqref{adjoint3_DE2}, Equation~\eqref{Hzero2}, Remark~\ref{rem:optcontr} and simple manipulations, along lines similar to those provided in the proof of Lemma~3 in~\cite{Kaya2017}.
\end{proof}

In a solution trajectory, we will denote a straight line segment (i.e., a singular arc, where $u(t) = 0$) by an $S$  and a circular arc segment of curvature $a$ (i.e., a nonsingular arc, where $u(t) = a$ or $-a$) by a $C$, resulting in descriptions of {\em optimal paths} to be {\em of type}, for example, $CSCC\cdots$, $SCS\cdots$, etc., representing concatenations of $S$ and $C$ type arcs.

In the rest of the paper, we will at times not show dependence of variables on $t$ for clarity of presentation.

\begin{remark}[Normal and Abnormal Phase Portraits] \rm \label{remark_int}
The differential equation in \eqref{lambda3_DE_int}, which is given in terms of the phase variables $\lambda_3$ and $\dot{\lambda}_3$, i.e., the switching function $\lambda_3$ and its derivative $\dot{\lambda}_3$, can be put into the form 
\begin{equation} \label{lambda3_DE_ellipse_int} 
\left(\lambda_3 \pm \frac{\lambda_0}{a}\right)^2 + \frac{\dot\lambda_3^2}{a^2} = \frac{\rho_i^2}{a^2}\,,
\end{equation}
when the optimal control is nonsingular, i.e., when $u(t) = \pm a\neq 0$.  Note that \eqref{lambda3_DE_ellipse_int} is akin to Equation (15) in~\cite{Kaya2017}, and the phase portrait for a given stage $i$, $i = 1,\ldots,N$, is depicted as in the case of the Markov--Dubins problem in~\cite{Kaya2017}: see the trajectories in Figure~\ref{phase_int}(a) for the normal case, $\lambda_0 > 0$, and Figure~\ref{phase_int}(b) for the abnormal case, $\lambda_0 = 0$.
\begin{figure}[t]
\begin{minipage}{90mm}
\begin{center}
\psfrag{L}{$\lambda_3$}
\psfrag{Ld}{\hspace*{-1mm}$\dot{\lambda}_3$}
\psfrag{b}{\footnotesize $-\lambda_0/a$}
\psfrag{c}{\footnotesize $\lambda_0/a$}
\psfrag{u1}{\footnotesize $u(t) = a$}
\psfrag{u2}{\footnotesize $u(t) = -a$}
\hspace*{-20mm}
\includegraphics[width=125mm]{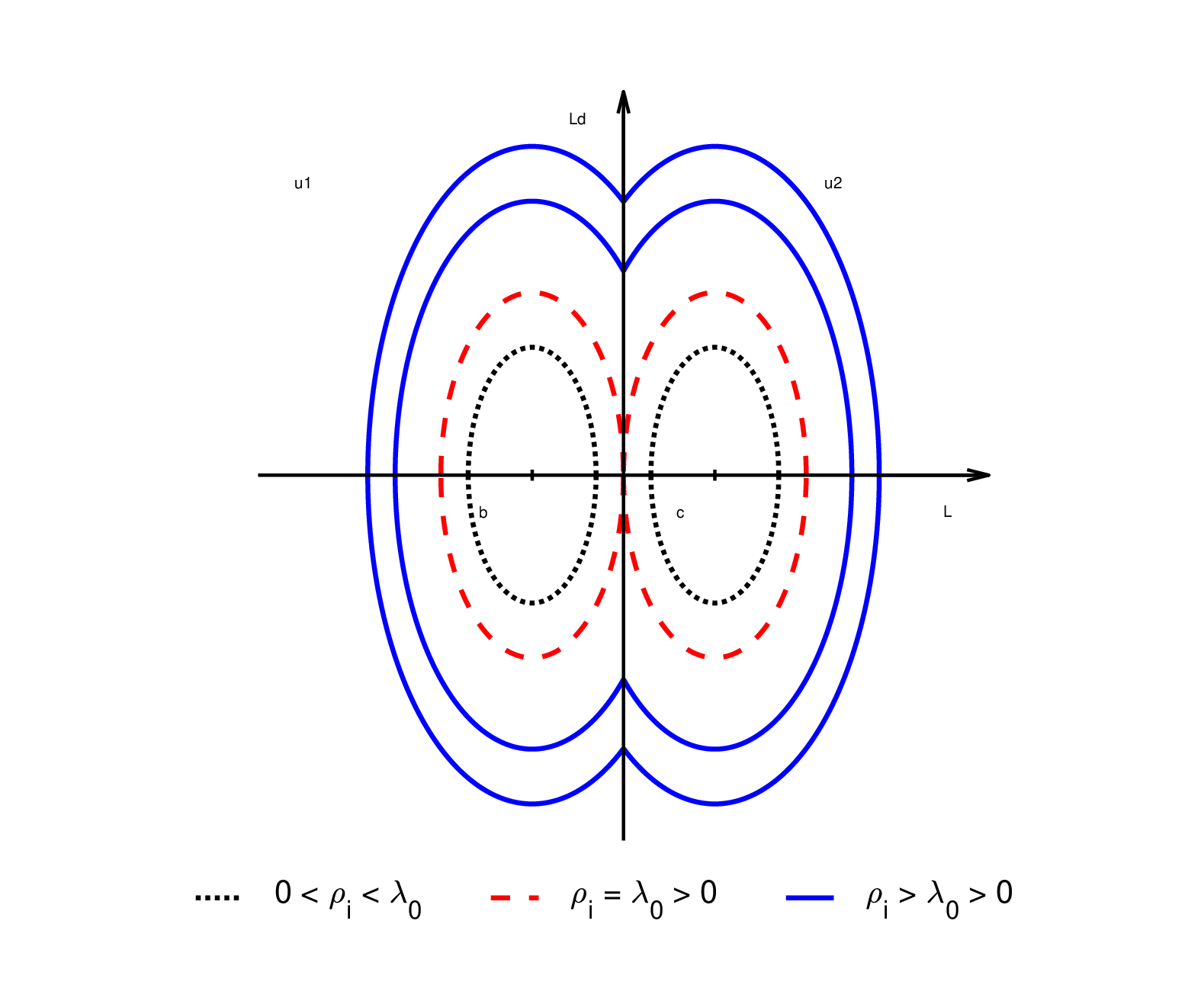} \\[-1mm]
{\small (a) The normal case, $\lambda_0 > 0$}
\end{center}
\end{minipage}
\hspace*{0mm}
\begin{minipage}{80mm}
\vspace*{0mm}
\begin{center}
\psfrag{L}{$\lambda_3$}
\psfrag{Ld}{\hspace*{-1mm}$\dot{\lambda}_3$}
\psfrag{u1}{\footnotesize $u(t) = a$}
\psfrag{u2}{\footnotesize $u(t) = -a$}
\hspace*{-10mm}
\includegraphics[width=80mm]{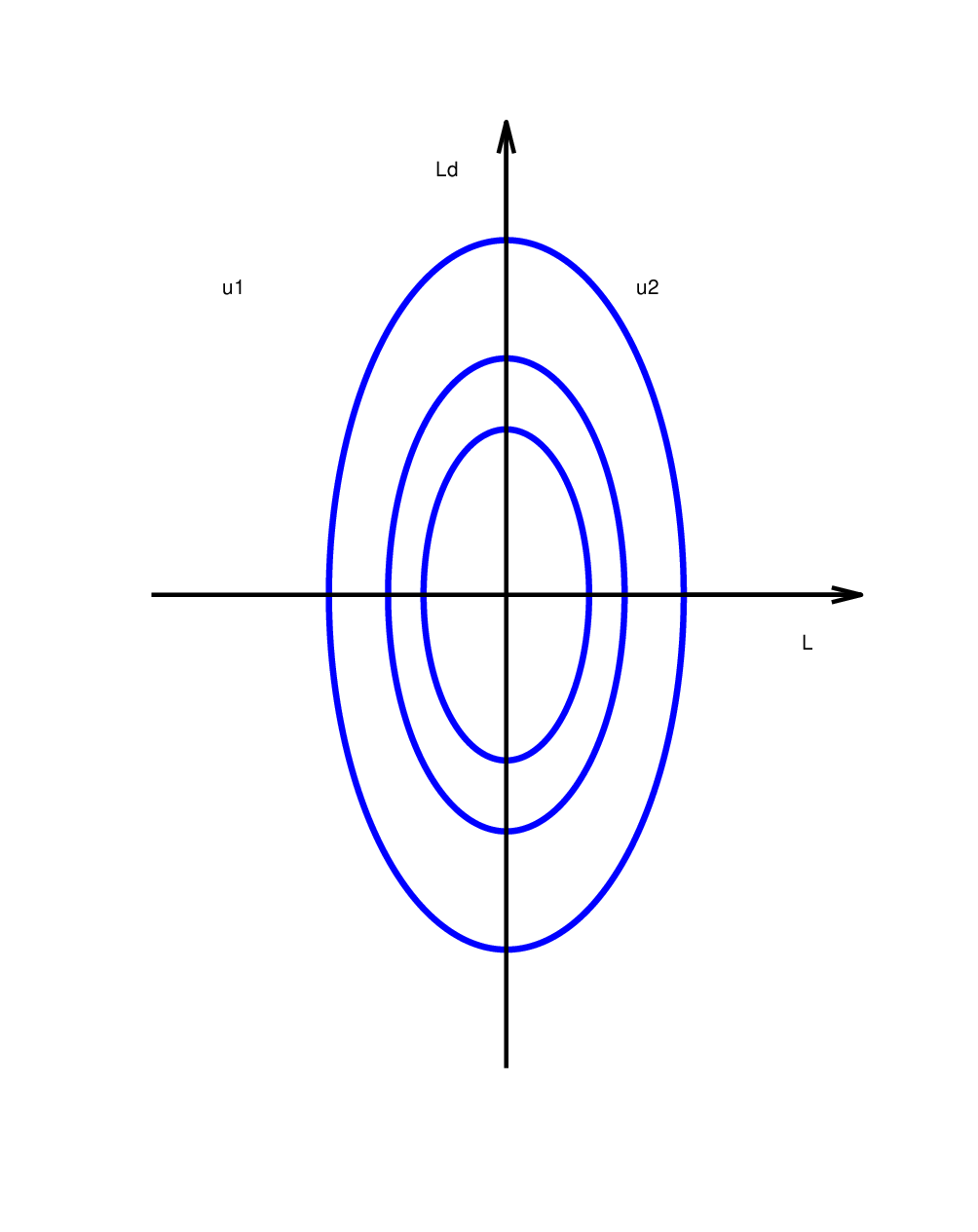} \\[-1mm]
{\small (b) The abnormal case, $\lambda_0 = 0$}
\end{center}
\end{minipage}
\
\caption{\small\sf Phase portrait of the differential equation in \eqref{lambda3_DE_int} in stage $i$.}
\label{phase_int}
\end{figure}

The phase portrait trajectories for the normal case in Figure~\ref{phase_int}(a) can be classified into three groups based on the relationship between $\rho_i$ and $\lambda_0$, in stage $i$, $i = 1,\ldots,N$:
\begin{itemize}
\item[(i)] $\rho_i > \lambda_0 > 0$: The optimal control is of bang--bang type, including only circular arcs, i.e., the optimal interpolating path in stage $i$ is of type $CC\cdots C$ or a subset thereof.  The phase portrait trajectories are concatenations of (pieces of) ellipses, examples of which are shown by (dark blue) solid curves.  It is not difficult to deduce from the portrait that if the bang--bang optimal control in a stage has two switchings then the second arc must have a length strictly greater than $\pi/a$.  The portrait, however, does not indicate how many switchings the bang--bang optimal control in a stage must have.
\item[(ii)] $\rho_i = \lambda_0 > 0$: The optimal control is of bang--singular--bang type, including circular arcs and straight lines such that the optimal interpolating path in stage $i$ is of type $CSCSC\cdots CSC$ or a subset thereof.  The phase portrait trajectory is represented by the two unique (red) dashed elliptic curves concatenated at the origin $(0,0)$.  Note that singular control, which corresponds to a straight line interpolant segment, takes place only at the origin $(0,0)$ of the phase plane. As in case~(i), the portrait does not indicate how many switchings the optimal control in a stage will have.
\item[(iii)] $0 < \rho_i < \lambda_0$: The optimal control is bang--bang, comprised of a single circular arc, i.e., the optimal interpolating path in stage $i$ is of type $C$. The phase portrait trajectories are ellipses, examples of which are shown by (black) dotted curves, along which, either $u(t) = a$ or $u(t) = -a$.  Since these ellipses never cross the $\dot{\lambda}_3$ axis, the interpolant path is of type $C$ in stage $i$.
\end{itemize}
Figure~\ref{phase_int}(b) depicts the phase portrait for the abnormal case, $\lambda_0 = 0$, in a stage, from which it is obvious to see that the abnormal optimal control is bang--bang.  Just like Figure~\ref{phase_int}(a), Figure~\ref{phase_int}(b) does not convey further information as to how many switchings the optimal control must have in a given stage.

In summary, for any given stage $i$, where $\rho_i$ is a constant associated with stage~$i$, elliptic trajectories in the phase portraits are suitably concatenated.  The phase plane trajectories pass through the origin if and only if the optimal path contains a straight line, i.e., if the optimal control in stage $i$ is singular over some interval $[\zeta_1,\zeta_2)\subset[t_{i-1},t_i)$.  Each sequence of concatenated elliptic curves, for all $t\in[t_{i-1},t_i)$, corresponds to a fixed $\rho_i$.  In general, $\rho_{i+1}\neq\rho_i$, and this corresponds to a "vertical" jump (up or down) to a different ellipse at $t=t_{i+1}$, with $\lambda_3(t_{i+1}) = \lambda_3(t_i)$ and $\dot{\lambda}_3(t_{i+1}) \neq \dot{\lambda}_3(t_i)$, in general.  
\endproof
\end{remark}

The following lemma is a companion of Lemma~4 in~\cite{Kaya2017}.
\begin{lemma} \label{nonsingular_int}
Suppose that optimal control $u(t)$ for Problem~{\em (Pc)} is nonsingular over a subinterval $[\zeta_3,\zeta_4)\subset[t_{i-1},t_i)$.  Then
\begin{equation} \label{lambda3_int} 
|\lambda_3(t)| = \frac{1}{a}\,\left[\rho_i\,\cos(\theta(t) - \phi_i) + \lambda_0\right],\mbox{\ \ a.e. } t\in[\zeta_3,\zeta_4)\subset[t_{i-1},t_i). 
\end{equation}
\end{lemma}
\begin{proof}
Substitution of $u(t) = -a\,\sgn(\lambda_3(t))$ and a re-arragement of the terms in \eqref{Hzero2} yield \eqref{lambda3_int}.
\end{proof}

The lemma given below is a companion of Lemma~5 in~\cite{Kaya2017}.
\begin{lemma}[Nonsingular Interpolant Segments]  \label{rho_nonsing_int}
Consider Problem~{\em (Pc)} and the necessary conditions of optimality for it.  \\[-6mm]
\begin{enumerate}
\item[(a)] If $\rho_i = 0$ for some $i = 1,\ldots,N$, then $\lambda_0 > 0$ and either $u(t) = a$ or $u(t) = -a$, for all $t\in[t_{i-1},t_i]$.
\item[(b)] If $\rho_i > 0$ for some $i = 1,\ldots,N$, and $\rho_i\neq\lambda_0$, then $\lambda_0 \ge 0$ and $u(t)$ is bang--bang type over the interval $[t_{i-1},t_i]$.
\end{enumerate}
\end{lemma}
\begin{proof}
The proof is furnished similarly as in Lemma~5 in~\cite{Kaya2017}: Part~(a) is proved using \eqref{uint}, \eqref{Hzero2}, and \eqref{lambda3_int} in Lemma~\ref{nonsingular_int}, and  part~(b) is proved using Lemma~\ref{singular_int}(a), and \eqref{lambda3_DE_int} in Lemma~\ref{lambda3_lemma_int}.
\end{proof}

\begin{remark}[Abnormal Interpolants] \rm \label{rem:abnormal}
From Lemma~\ref{rho_nonsing_int}(a), if $\rho_i = 0$ for some $i = 1,\ldots,N$,  then the problem is normal and the optimal control is the constant value $a$ or $-a$.  Otherwise, if $\rho_i \neq 0$ for some $i = 1,\ldots,N$, then, by Lemma~\ref{rho_nonsing_int}(b), an abnormal solution is entirely possible, i.e., one might have that $\lambda_0 = 0$.  Since $\lambda_0$ is the same value in each stage, if $\lambda_0 = 0$, then the solution is abnormal in every single stage.  On the other hand, if the interpolating curve has a straight line segment at any stage, then, by Lemma~\ref{singular_int}(a), the whole curve has to be normal.  Note that, in the abnormal case, Equation~\eqref{lambda3_DE_int} reduces to
\begin{equation} \label{lambda3_abnormal_int} 
a^2\,\lambda_3^2(t) + \dot{\lambda}_3^2(t) = \rho_i^2\,,\mbox{ for all } t\in[t_{i-1},t_i)\,,\ i = 1,\ldots,N\,. 
\end{equation}
So, Figure~\ref{phase_int}(b) illustrates the phase portrait of $\lambda_3$ by the concentric ellipses for a single stage.  It must however be noted that, for different values of $\rho_i$, the trajectories in the phase plane will lie in a different ellipse in each stage $i$, with jumps from one ellipse to another at the junctions/nodes, i.e., at $t_i$, $i=1,\ldots,N-1$.
\endproof
\end{remark}

The following lemma is a companion of Lemma~6 in~\cite{Kaya2017}.
\begin{lemma}[Straight Line Interpolant Segments] \label{straight_int} 
Consider Problem~{\em (Pc)}.  If an optimal path over the interval $[t_{i-1},t_i]$ in stage~$i$, $i = 1,\ldots,N$, contains a straight line segment $S$, then it is of type $CSC$, $CS$, $SC$ or $S$.
\end{lemma}
\begin{proof}
Note that the lemma is given for a single stage, and the phase plane diagram of $\lambda_3$ in this stage is given in Figure~\ref{phase_int}(a).  So, we have the same setting as that for Lemma~6 in~\cite{Kaya2017}.  Therefore, the proof can be furnished along lines similar to those in Lemma~6 in~\cite{Kaya2017}, this time by using Figure~\ref{phase_int}(a) and Remark~\ref{remark_int}.
\end{proof}

Part~(a) of the lemma below is a companion of Lemma~7 in~\cite{Kaya2017}.
\newpage
\begin{lemma}[Abnormal Markov--Dubins Interpolating Curves] \label{abnormal_int} \ \\[-6mm]
\begin{enumerate}
\item[(a)]  An abnormal optimal path for Problem~{\em (Pc)} over the interval $[t_{i-1},t_i]$ in stage $i$,\linebreak $i = 1,\ldots,N$, is either of type $CC$ or $C$, with respective lengths of at most $2\pi/a$ and $\pi/a$.
\item[(b)]  An abnormal optimal path for Problem~{\em (Pc)} is of type $CC\cdots C$, with at least $N$ and at most $2N$ copies of $C$ concatenated, resulting in the length of the path to be at most $2N\pi/a$.
\end{enumerate}
\end{lemma}
\begin{proof}
Recall that for an abnormal path in a stage, $\lambda_0 = 0$. \\
(a) By Lemma~\ref{singular_int}(a), the optimal control is of bang--bang type, and the phase diagram of $\lambda_3$ in that stage is given as in Figure~\ref{phase_int}(b). The proof is then furnished similarly as in the proof of Lemma~7 in~\cite{Kaya2017}, which uses the diagram in Figure~3 in~\cite{Kaya2017}, but with $p_{i-1}$ and $p_i$ replaced by $z_0$ and $z_f$, respectively. \\
(b) With $\lambda_0 = 0$, the path in all stages is abnormal and thus, by Lemma~\ref{abnormal_int}(a), the path in each stage will be either of type $C$ or $CC$, facilitating the first part of the conclusion.  Since there are $N$ stages, one would have at least $N$ and at most $2N$ copies of $C$ concatenated, which results in the total length of the path to be at most $2N\pi/a$, completing the proof.
\end{proof}

The following lemma is a companion of Lemma~8 in~\cite{Kaya2017}.
\begin{lemma}[Non-optimality of a \boldmath$CCCC$-type curve in a stage]  \label{CCCC_int} 
Consider Problem~{\em (Pc)}.\linebreak  Any path of type CCCC over the interval $[t_{i-1},t_i]$ in stage $i$, $i = 1,\ldots,N$, is not optimal.
\end{lemma}
\begin{proof}
If the optimal path in a stage is abnormal, then, by Lemma~\ref{abnormal_int}(a), the statement holds immediately.  Suppose that Problem~(Pc) is normal.  Then a general configuration for a {\em candidate} optimal path which is of type $CCCC$ in stage $i$ will be as shown in Figure~4 in~\cite{Kaya2017}, since the phase diagram of $\lambda_3$ in that stage will be as given in Figure~\ref{phase_int}(a), with the lengths of each of the second and third circular arcs being $\pi+\gamma$, where $\gamma > 0$.  The rest of the proof is the same as that of Lemma~8 in~\cite{Kaya2017}.
\end{proof}

Next, we provide a companion of Theorem~1 (Dubins' theorem) in~\cite{Kaya2017}, for the Markov--Dubins interpolation problem.  It reduces to Dubins' theorem~\cite{Dubins1957,Kaya2017} for the interpolating curve segments between any two given consecutive points.
\begin{theorem}[Markov--Dubins Interpolating Curves]  \label{Dubins_int}
Any solution of Problem~{\em (P)}, that is, any $C^1$ and piecewise-$C^2$ shortest path of constrained curvature in the plane between two prescribed endpoints such that the slopes at the endpoints are prescribed and the path visits a sequence of intermediate points, is of type $CSC$, or of type $CCC$, or a subset thereof, between any two consecutive points.  Moreover, if the shortest path is of type $CCC$ between two consecutive points, then the second circular arc is of length greater than $\pi/a$.
\end{theorem}
\begin{proof}
Consider Stage~$i$ of the solution curve, $i = 1,\ldots,N$. If the solution is abnormal, i.e., $\lambda_0 = 0$, then by Lemma~\ref{abnormal_int}(a) the shortest path in stage $i$ is of type $C$ or $CC$, which in either case is a subarc of $CCC$.  Suppose that the solution is normal, i.e., $\lambda_0 \neq 0$.  Then the shortest path in stage $i$ is either of type
\begin{enumerate}
\item[(i)] $CSC$, $CS$, $SC$ or $S$, if it contains a straight line segment, by Lemma~\ref{straight_int}, or
\item[(ii)] $CCC$, $CC$ or $C$, by Lemmas~\ref{rho_nonsing_int} and~\ref{CCCC_int}.
\end{enumerate}
The last statement of the theorem is proved by using the phase plane diagram in Figure~\ref{phase_int}(a), with $\rho>\lambda_0>0$, for stage $i$.  If the shortest path is of type $CCC$, then three pieces of ellipses in Figure~\ref{phase_int}(a) are concatenated, with the second ellipse sweeping an angle greater than $\pi$, completing the proof.
\end{proof}

\section{Number of Subarcs and Stationarity of Feasible Solutions}
\label{sec:stationarity}

When it comes to computations, it is desirable to be able to say something more about the overall structure of a solution curve of Problem~(P), in terms of the subarcs $C$ and $S$.  By Theorem~\ref{Dubins_int} above, since in each stage one can have at most  three of the $C$ and $S$ subarcs in total, an upper bound on the combined number of $C$ and $S$ subarcs along an interpolating curve is simply $3N$.  Continuity of the adjoint variable, or the switching function, $\lambda_3$, for Problem~(Pc), which amounts to continuity of the signed curvature, will allow us to sharpen this bound slightly further.  

Suppose that a sequence of points/nodes are given at random for Markov--Dubins interpolation.  In view of computations, it is a very rare occasion that the signed curvature of the interpolating curve switches from $a$ to $-a$, or vice versa, at any of the given nodes.  Under the assumption that this rare event does not occur, Proposition~\ref{subarcs_int} below states that the total number of $C$ and $S$ subarcs is at most $2N+1$, which is considerably smaller than $3N$ when $N$ is large.
\begin{proposition}[Number of Subarcs]  \label{subarcs_int}
If the signed curvature of an optimal path for Problem~{\em (Pc)} does not switch between $a$ to $-a$ at the nodes, i.e., if $u(t_i^+) = u(t_i^-)$, for all $i = 1,\ldots,N-1$, then the total number of $C$ and $S$ subarcs in the optimal path is at most $2N+1$.
\end{proposition}
\begin{proof}
Consider stage $i$, $i = 1,\ldots,N-1$.  Note that, by continuity of $\lambda_3$, $\lambda_3(t_i^+) = \lambda_3(t_i^-) = \lambda_3(t_i)$.  Suppose that $u(t_i^+) = u(t_i^-)$.  Then, by $u(t) = a\,\sgn(\lambda_3(t))$,  $\sgn(\lambda_3(t_i-\varepsilon)) = \sgn(\lambda_3(t_i+\varepsilon))$ for all small enough $\varepsilon > 0$.  There are two cases to consider.
\begin{enumerate}
\item[(i)] $\lambda_3(t_i) \neq 0$:  In this case, $u(t_i-\varepsilon) = u(t_i+\varepsilon) =  a$ or $-a$ for all small $\varepsilon > 0$, which means that we have the same $C$ subarc immediately before and immediately after the node and so the total number of arcs in stages $i$ and $i+1$ are reduced from at most six to at most five.  In other words, the bound on the total number of subarcs has been reduced by one thanks to node~$i$.
\item[(ii)] $\lambda_3(t_i) = 0$:  In this case, for all small $\varepsilon > 0$, the pair $(u(t_i-\varepsilon), u(t_i+\varepsilon))$ has one of the values $(0,\pm a)$, $(\pm a, 0)$ and $(0,0)$, which correspond to the subarc pairs $SC$, $CS$ and $SS$, respectively.  In the case of $SC$, the path in stage~$i$ will be of type $CS$ or $S$ or a subset thereof by Theorem~\ref{Dubins_int}, and the path in stage~$(i+1)$ will be of type $CSC$ or $CCC$ or a subset thereof.  As in part~(i) above, the maximum number of subarcs in the two consecutive stages $i$ and $(i+1)$ is reduced from six to five.  For the case of $CS$, symmetric arguments can be used to get the same conclusion.  In the case of $SS$, again by Theorem~\ref{Dubins_int}, the path in stage~$i$ will be of type $CS$ or $S$, and the path in stage~$(i+1)$ will be of type $SC$ or $S$; so, the maximum total number of subarcs in the two consecutive stages $i$ and $i+1$ is reduced from four to three.
\end{enumerate}
In both of the cases~(i) and (ii) above, the maximum number of subarcs in the two consecutive stages $i$ and $(i+1)$, $i = 1,\ldots, N-1$, is reduced from six to five or three, depending on the combinations of the subarcs.  In other words, the bound on the total number of subarcs has been reduced by one about each node~$i$.  Summation over each node gives $N-1$, and so one gets $3N - (N-1) = 2N + 1$, as asserted.
\end{proof}

\begin{remark}  \rm
Suppose that one of the subarcs in stage $i$ is singular, that is of type $S$.  Then, recall by Lemma~\ref{singular_int}(a) that, $\rho_i = \lambda_0$.  So one gets $\cos(\theta(t) - \phi_i) = -1$.  Let $\thetab_i := \theta(t)$, a real constant, along the straight line subarc.  Then
\[
\phi_i = \thetab_i - \pi\,,
\]
and, after algebraic manipulations,
\[
\lambda_1^i = \frac{\lambda_0}{\sqrt{1 + \tan^2(\thetab_i - \pi)}}\qquad\mbox{and}\qquad
\lambda_2^i = \frac{\lambda_0}{\sqrt{1 + \cot^2(\thetab_i - \pi)}}\,.
\]
\vspace*{-5mm}\endproof
\end{remark}
Without loss of generality, and for simplicity, one can take $\lambda_0 = 1$.

In Theorem~\ref{stationarity_int} below, we state that feasible solutions of Problem~(Pmc), which in each stage are of  the types $CSC$ or $CCC$, or a subset thereof, verify the necessary conditions of optimality, i.e., they are also {\em stationary}, or {\em critical}, {\em solutions} of Problem~(Pc), if there exists a solution to a certain system of equalities and inequalities for $\rho_i$ and $\phi_i$, for all $i = 1,\ldots,N$.

In the proof of Theorem~\ref{stationarity_int} below, we assume that a feasible interpolating curve, which is of these certain types in each stage, has been provided.  In other words, the times $t_1^i$ and $t_2^i$ at which switchings from one subarc to another occur in stage $i$ (note that in general $t_{i-1}\le t_1^i\le t_2^i\le t_i$), as well as the terminal time $t_i$ of stage $i$, which is the length of the curve in stage $i$, $i = 1,\ldots,N$, are at hand.  Therefore the signed curvatures ($a$ or $-a$) of the $C$ subarcs of the feasible curve are given, so, 
\[
\theta_1^i := \theta(t_1^i)\quad\mbox{and}\quad \theta_2^i := \theta(t_2^i)
\]
are also known/easily calculable.  Also note that 
\[
\theta_0^1 := \theta_0\quad\mbox{and that}\quad \theta_0^i := \theta(t_{i-1})\,,\quad i = 2,\ldots,N\,,
\]
with $t_0 := 0$.
\begin{theorem}[Normal Stationarity of Feasible Interpolating Curves]  \label{stationarity_int}
Any feasible solution of Problem~{\em (Pmc)}, i.e., any path satisfying the constraints of Problem~{\em (Pmc)}, which in stage $i$ is of type $CSC$ or $CCC$, or a subset thereof, verifies the maximum principle, if the system of equalities and inequalities in \eqref{nodes}--\eqref{segments} below has a solution for $\rho_i$ and~$\phi_i$, for all $i = 1,\ldots,N$, with $\lambda_0 = 1$: \\[1mm]
For each intermediate node $i$, $i = 1,\ldots,N - 1$,
\begin{equation}  \label{nodes}
\left.\begin{array}{ll}
\rho_{i+1}\,\cos(\theta(t_i) - \phi_{i+1}) = \rho_i\,\cos(\theta(t_i) - \phi_i)\,, & \mbox{ if } u(t_i^+) = u(t_i^-)\,, \\[2mm]
\rho_{i+1}\,\cos(\theta(t_i) - \phi_{i+1}) = -\rho_i\,\cos(\theta(t_i) - \phi_i)\,, & \mbox{ if } u(t_i^+) = -u(t_i^-)\,,
\end{array}\right\}
\end{equation}
and, along each curve segment in stage $i$, $i = 1,\ldots,N$,
\begin{equation}  \label{segments}
\left.\begin{array}{ll}
\rho_i = 1\,,\ \ \phi_i = \theta_1^i \pm \pi\,, & \mbox{ for types } CSC, CS, SC \mbox{ and } S\,, \\[2mm]
\rho_i = -\sec\left((\theta_1^i - \theta_2^i)/2\right)\,,\ \ \phi_i = (\theta_1^i + \theta_2^i)/2\,, & \mbox{ for type }  CCC\,,  \\[2mm]
\rho_i > 1\,,\ \ \phi_i = \theta_1^i - \cos^{-1}(-1/\rho_i)\,, & \mbox{ for type }  CC\,,  \\[2mm]
\rho_i > 0\,,\ \ -\pi < \phi_i < \pi\,, & \mbox{ for type }  C\,.\\[2mm]
\end{array}\right\}
\end{equation}
\end{theorem}
\begin{proof}
Continuity of the adjoint variable~$\lambda_3$ at node $i$ implies from \eqref{lambda3_int}, or~\eqref{Hzero2}, that
\[
\rho_{i+1}\,\cos(\theta(t_i) - \phi_{i+1}) = \rho_i\,\cos(\theta(t_i) - \phi_i)\,,
\]
if $u(t_i^+) = u(t_i^-)$ (as in Proposition~\ref{subarcs_int}), or that
\[
\rho_{i+1}\,\cos(\theta(t_i) - \phi_{i+1}) = -\rho_i\,\cos(\theta(t_i) - \phi_i)\,,
\]
if $u(t_i^+) = -u(t_i^-)$ (which is the case excluded in Proposition~\ref{subarcs_int}).  So, \eqref{nodes} is furnished as required.

Consider, in stage $i$, a feasible curve segment of type $CSC$, or of one of the types $CS$, $SC$, and $S$.  Recall that along the subarc $S$, $\lambda_3(t) = 0$, and so $\rho_i = \lambda_0 > 0$ by Lemma~\ref{singular_int}, and one can set, without loss of generality, $\rho_i = \lambda_0 = 1$.  Hence, Equation~\eqref{Hzero2} reduces to $\cos(\theta(t) - \phi_i) = -1$, with $\theta(t) = \theta_1^i$ constant, which implies that $\phi_i = \theta_1^i \pm \pi$.  These provide the first line of expressions in~\eqref{segments}.

In stage $i$, consider feasible curves of types $CCC$, $CC$ and $C$, one by one, and derive the related expressions in \eqref{segments} as follows.
\begin{itemize}
\item[(a)] Type $CCC$:  This type requires two switchings; so, $t_{i-1} < t_1^i < t_2^i < t_i$ and that $\lambda_3(t_1^i) = 0$ and $\lambda_3(t_2^i) = 0$, with which Equation~\eqref{Hzero2} yields two equations in the two unknowns $\rho_i$ and $\phi_i$; namely, $\rho_i\,\cos(\theta_1^i - \phi_i) + 1 = 0$ and $\rho_i\,\cos(\theta_2^i - \phi_i) + 1 = 0$, where $\rho_i > 1$ by Remark~\ref{remark_int}.  These two equations result in $\cos(\theta_1^i - \phi_i) = \cos(\theta_2^i - \phi_i)$.  By Figure~\ref{phase_int}(a) and the second statement of Theorem~\ref{Dubins_int}, $\theta_2^i - \phi_i = -(\theta_1^i - \phi_i)$.  Then simple algebraic manipulations provide a unique solution for the constants $\rho_i$ and $\phi_i$ as:
\[
\phi_i = (\theta_1^i + \theta_2^i)/2\,\quad\mbox{and}\quad
\rho_i = -\sec\left((\theta_1^i - \theta_2^i)/2\right)\,.
\]

\item[(b)]  Type $CC$: This type requires only one switching; so, without loss of generality, let $t_{i-1} < t_1^i < t_2^i = t_i$.  Then $\lambda_3(t_1^i) = 0$ and Equation~\eqref{Hzero2} result in $\rho_i\,\cos(\theta_1^i - \phi_i) + 1 = 0$, i.e., $\phi_i = \theta_1^i - \cos^{-1}(-1/\rho_i)$, where $\rho_i > 1$ by Remark~\ref{remark_int}.

\item[(c)]  Type $C$:  This type requires no switchings, so by Remark~\ref{remark_int}, any $\rho_i > 0$ and any $-\pi\le\phi_i\le\pi$ would do.  \\[-10mm]
\end{itemize}
\end{proof}

The following theorem is a companion of Theorem~\ref{stationarity_int}.  It states that under certain junction/node conditions any feasible interpolating curve of Problem~(Pmc) which is of type $CC$ or $C$ between each two consecutive nodes, and each $C$-subarc in the curve is of length not greater than $\pi/a$, is stationary, with $\lambda=0$, i.e., abnormal.

\begin{theorem}[Abnormal Stationarity of Feasible Interpolating Curves]  \label{stationarity_int_abnormal}
Any feasible path for Problem~{\em (Pmc)}, i.e., any path satisfying the constraints of Problem~{\em (Pmc)}, which in stage $i$ is of type $CC$ or $C$, where the length of any subarc $C$ is not greater than $\pi/a$, verifies the maximum principle, if the system of equalities and inequalities in \eqref{nodes_abnormal}--\eqref{segments_abnormal} below has a solution for $\rho_i$ and~$\phi_i$, for all $i = 1,\ldots,N$, with $\lambda_0 = 0$: \\[1mm]
For each intermediate node $i$, $i = 1,\ldots,N - 1$,
\begin{equation}  \label{nodes_abnormal}
\left.\begin{array}{ll}
\rho_{i+1}\,\cos(\theta(t_i) - \phi_{i+1}) = \rho_i\,\cos(\theta(t_i) - \phi_i)\,, & \mbox{ if } u(t_i^+) = u(t_i^-)\,, \\[2mm]
\rho_{i+1}\,\cos(\theta(t_i) - \phi_{i+1}) = -\rho_i\,\cos(\theta(t_i) - \phi_i)\,, & \mbox{ if } u(t_i^+) = -u(t_i^-)\,,
\end{array}\right\}
\end{equation}
and, along each curve segment in stage $i$, $i = 1,\ldots,N$,
\begin{equation}  \label{segments_abnormal}
\left.\begin{array}{ll}
\rho_i > 0\,,\ \ \phi_i = \theta_1^i \pm \pi/2\,, & \mbox{ for type }  CC\,,  \\[2mm]
\rho_i > 0\,,\ \ \phi_i = \theta_0^i + u(t_{i-1}^+)\,\pi/2\,, & \mbox{ for type }  C\,.\\[2mm]
\end{array}\right\}
\end{equation}
\end{theorem}
\begin{proof}
In stage $i$, consider feasible curves of types $CC$ and $C$, one by one, and derive the pertaining expressions in \eqref{segments_abnormal} as follows.\begin{itemize}
\item[(a)]  Type $CC$:  This type has only one switching, so without loss of generality, let $t_{i-1}\neq t_1^i\neq t_2^i = t_i$.  Then $\lambda_3(t_1^i) = 0$ and Equation~\eqref{Hzero2} give $\rho_i\,\cos(\theta_1^i - \phi_i) = 0$, where $\rho_i > 0$ by Remark~\ref{rem:abnormal}, and so $\cos(\theta_1^i - \phi_i) = 0$, which yields $\phi_i = \theta_1^i \pm \pi/2$.
\item[(b)]  Type $C$:  This type has no switchings; so, $\rho_i > 0$ and $\theta_0^i-\phi_i = -\sgn(u(t_{i-1}^+))\,\pi/2$.  \\[-10mm]
\end{itemize}
\end{proof}

In the following corollary to Theorem~\ref{stationarity_int}, we claim that Equations~\eqref{nodes}--\eqref{segments} are readily satisfied for feasible interpolating curves with stages of type $CSC$.  Note that the case when $u(t_i^+) = -u(t_i^-)$ is not interesting, as otherwise the second $C$-subarc in any stage is a full circle, which obviously is not optimal.  Therefore we only consider the case when $u(t_i^+) = u(t_i^-)$.

\begin{theorem}[Stationarity of Interpolating curves with stages of type CSC]  \label{stationary_CSC}
Any feasible solution of Problem~{\em (Pmc)}, i.e., any path satisfying the constraints of Problem~{\em (Pmc)}, which in every stage is of type $CSC$, verifies the maximum principle, with $u(t_i^+) = u(t_i^-)$, $i = 1,\ldots,N-1$.
\end{theorem}
\begin{proof}
Suppose that an interpolating curve with stages of type $CSC$ is a feasible solution of Problem~(Pmc). By Lemma~\ref{singular_int} and Remark~\ref{remark_int}(ii), $\rho_i = \rho_{i+1} = \lambda_0 > 0$.  Without loss of generality, set $\lambda_0 = 1$.  Figure~\ref{CSC_phase} reproduces the phase diagram of $\lambda_3$ and $\dot{\lambda}_3$, earlier shown in Figure~\ref{phase_int}(a), for the particular case of $CSC$ type arcs, i.e., when $\rho = \lambda_0 = 1$.  The particular instance considered in the diagram is one where the third subarc in the $i$th stage is an $R$-subarc; however, the case when the third subarc is an $L$-subarc can be treated in a similar fashion.
\begin{figure}[t]
\psfrag{L}{$\lambda_3$}
\psfrag{Ld}{\hspace*{-1mm}$\dot{\lambda}_3$}
\psfrag{b}{\footnotesize $-1/a$}
\psfrag{c}{\footnotesize $1/a$}
\psfrag{u1}{\footnotesize $u(t) = a$}
\psfrag{u2}{\footnotesize $u(t) = -a$}
\psfrag{theta}{\footnotesize $\theta(t_i)-\phi_i$}
\psfrag{negtheta}{\footnotesize $-(\theta(t_i)-\phi_i)$}
\psfrag{A}{\small A}
\psfrag{B}{\small B}
\psfrag{p1}{\footnotesize $1$}
\psfrag{m1}{\footnotesize $-1$}
\[
\includegraphics[width=100mm]{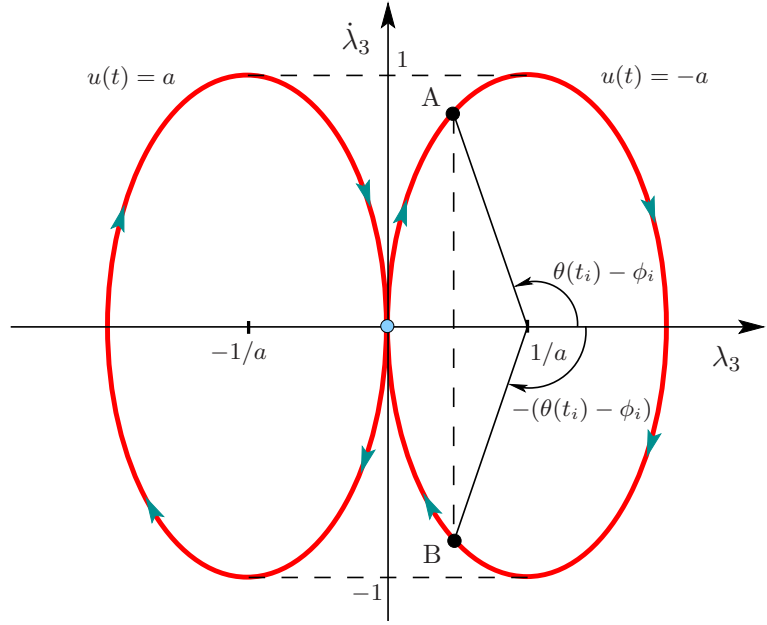}
\]
\caption{\small\sf Phase diagram for the proof of Theorem~\ref{stationary_CSC}.}
\label{CSC_phase}
\end{figure}
In the diagram, the point labelled as A represents a node where one goes from Stage $i$ to Stage $(i+1)$ of the interpolating curve.  By continuity of $\lambda_3$, $u(t_i^+) = u(t_i^-)$.  So, in Equations~\eqref{nodes}, we only consider the case
\begin{equation}  \label{cos_eqn_CSC}
\cos(\theta(t_i) - \phi_{i+1}) = \cos(\theta(t_i) - \phi_i)\,,
\end{equation}
$i=1,\ldots,N-1$.  Clearly, at the origin, where switchings, first from $C$ to $S$ and then $S$ to $C$, occur, $\theta(t) = \theta_1^i$, for all $t\in[t_1^i,t_2^i]$, and, from the diagram,
\begin{equation} \label{phi_i}
\theta_1^i - \phi_i = -\sgn(u(t_i))\,\pi\,,
\end{equation}
in the $i$th stage.  Therefore, the first (or the relevant) condition in \eqref{segments} is satisfied in the $i$th stage.  It suffices to show next that \eqref{segments} is satisfied in the $(i+1)$st stage.

Equation~\eqref{cos_eqn_CSC} implies two cases: \\[2mm]
(i) $\theta(t_i) - \phi_{i+1} = \theta(t_i) - \phi_i$: In this case, $\dot{\lambda}_3$ is continuous, and so, as can be seen from the diagram, the $C$ subarc (an $R$-subarc with $u(t)=-a$) becomes a whole circle before a switching to an $S$ subarc occurs at the origin. \\[2mm]
(ii) $\theta(t_i) - \phi_{i+1} = -(\theta(t_i) - \phi_i)$:  In this case, $\dot{\lambda}_3$ is discontinuous, in that the phase plane trajectory jumps from point A to point B, as can be seen in Figure~\ref{CSC_phase}. \\[2mm]
Since in Stage $(i+1)$ the curve is also of type $CSC$, one has 
\begin{equation} \label{phi_i+1}
\theta_1^{i+1} - \phi_{i+1} = \sgn(u(t_i))\,\pi\,,
\end{equation}
as can be observed in Figure~\ref{CSC_phase}, satisfying~\eqref{segments}.
\end{proof}

\newpage
\begin{proposition}  \label{lengths_CSC}
If the interpolating curve solving Problem~(P) is of type $CSC$ in each of the two consecutive stages $i$ and $(i+1)$, then
\begin{itemize}
\item[{\rm (a)}] $|\theta(t_i) - \theta_1^i| < \pi$, and
\item[{\rm (b)}] $\theta(t_i) = (\theta_1^i + \theta_1^{i+1}) / 2$\,.
\end{itemize}
\end{proposition}
\begin{proof}
Suppose that the (optimal) interpolant is of type $CSC$ in each of the $i$th and $(i+1)$st stages.  The proofs of parts (a) and (b) are provided separately as follows. \\[2mm]
(a) Suppose that $|\theta(t_i) - \theta_1^i| \ge \pi$.  Then $\theta(t_i) - \phi_i \ge\pi$.  Using the conclusion $u(t_i^+) = u(t_i^-)$, $i = 1,\ldots,N$, of Theorem~\ref{stationary_CSC}, $CSC|CSC$ reduces to $CSCSC$, where the $C$ subarc in the middle which, by Figure~\ref{CSC_phase}, is of length at least $2\pi$, resulting in a full circular subarc, which cannot be optimal, yielding a contradiction. \\[2mm]
(b) Since the case (i) in the proof of Theorem~\ref{stationary_CSC} yields a full circular subarc, and so is non-optimal, we consider only the case (ii) in which $\theta(t_i) - \phi_{i+1} = -(\theta(t_i) - \phi_i)$, i.e.,
\begin{equation}  \label{2theta}
2\theta(t_i) = \phi_i + \phi_{i+1}\,.
\end{equation}
Then substitutions of $\phi_i$ and $\phi_{i+1}$ in \eqref{phi_i} and \eqref{phi_i+1} into \eqref{2theta} and a re-arrangement of the terms yield the required result.
\end{proof}

\begin{remark}  \label{midpoint_CSC} \rm
Proposition~\ref{lengths_CSC}(b) implies that the node (where one goes from stage $i$ to stage $(i+1)$) is the midpoint of the $C$ subarc, i.e., it subdivides the $C$ subarc into two segments of equal length.  This provides both a necessary condition of optimality of a Markov--Dubins interpolating curve and a computational tool which might be employed to facilitate/speed up convergence.  For example, in the case when consecutive nodes are sufficiently far from one another, $CCC$ cannot be an option and so $CSC$ is the type of stage arcs one should contemplate.  In such a case, one can incorporate the fact that the nodes will be the midpoints of the respective intermediate $C$ subarcs.
\endproof
\end{remark}

\section{A Numerical Method and Experiments}
\label{methods}

\subsection{A Numerical method for Markov--Dubins interpolating curves}
\label{themethod}

In this section, the numerical technique presented in~\cite{Kaya2017} for finding Markov--Dubins curves, based on switching time optimization, or arc parameterization,  will be generalized for finding Markov--Dubins interpolating curves.  We adopt a similar terminology as that in~\cite{Kaya2017} in terms of the stages of the interpolant curves.  Define the subarc lengths for stage $i$, $i = 1,\ldots,N$, as
\begin{equation}  \label{subarc_lengths}
\xi_j ^i:= t^i_j - t^i_{j-1}\,,\quad\mbox{for } j = 1,\ldots,5\,,
\end{equation}
where $t^i_j$ are the {\em switching times} for the subarcs in stage $i$.  Let $t^1_0 := 0$ and $t^N_5 := t_N$.  The notation and terminology in this section come from earlier work on arc parameterization, or switching time optimization, studied in~\cite{KayNoa1994,KayNoa1996,KayLucSim2004,KayNoa2003,MauBueKimKay2005,KayMau2014} for problems whose solutions cannot be derived analytically, unlike the problem we have here.  

We represent the possible types of concatenated subarc solutions throughout all of the stages sequentially as
\[
L_{\xi^1_1}R_{\xi^1_2}S_{\xi^1_3}L_{\xi^1_4}R_{\xi^1_5}\ |\ L_{\xi^2_1}R_{\xi^2_2}S_{\xi^2_3}L_{\xi^2_4}R_{\xi^2_5}\ |\ \cdots\ |\ L_{\xi^N_1}R_{\xi^N_2}S_{\xi^N_3}L_{\xi^N_4}R_{\xi^N_5}\,,
\]
where $L$ (left-turn), $R$ (right-turn) and $S$ (straight-line) and the associated notation used here is the same as those defined in~\cite{Kaya2017}.  Although, formally, five subarcs are concatenated in each stage, at most three of the arc durations can be nonzero in an optimal solution.  Recall indeed that, by Theorem~\ref{Dubins_int}, in any given stage the path will be of type $CSC$ or $CCC$, or a subset of these strings.  Here, $C$ can be represented either by $L$ (a left-turn arc) or $R$ (a right-turn arc).  For example, the type $RLR$ in the $i$th stage is given by $\xi^i_1=\xi^i_3=0$ and $\xi^i_2,\xi^i_4,\xi^i_5 > 0$.

The solution of the ODEs in Problem~(Pc) can be given as follows.  For $t^i_{j-1} \le t < t^i_j$, and all $i = 1,\ldots,N$,
\begin{eqnarray}
&& \theta(t) = \theta(t^i_{j-1}) + u(t)\,(t - t^i_{j-1})\,,\quad\mbox{ if } j = 1,\ldots,5\,, \label{theta_arc_int} \\[2mm]
&& x(t) = \left\{\begin{array}{ll}
x(t^i_{j-1}) + (\sin\theta(t) - \sin\theta(t^i_{j-1})) / u(t)\,, & \mbox{ if } j = 1,2,4,5\,, \\[2mm]
x(t^i_{j-1}) + \cos\theta(t)\,(t - t^i_{j-1})\,, & \mbox{ if } j = 3\,,
\end{array}\right. \label{x_arc_int} \\[3mm]
&& y(t) = \left\{\begin{array}{ll}
y(t^i_{j-1}) - (\cos\theta(t) - \cos\theta(t^i_{j-1})) / u(t)\,, & \mbox{ if } j = 1,2,4,5\,, \\[2mm]
y(t^i_{j-1}) + \sin\theta(t)\,(t - t^i_{j-1})\,, & \mbox{ if } j = 3\,, \label{y_arc_int}
\end{array}\right.
\end{eqnarray}
where
\begin{equation}  \label{arc_control}
u(t) = \left\{\begin{array}{rl}
a\,, & \mbox{ if } j = 1,4\,, \\[1mm]
-a\,, & \mbox{ if } j = 2,5\,, \\[1mm]
0\,, & \mbox{ if } j = 3\,.
\end{array}\right.
\end{equation}
We note that the control variable $u(t)$ is a piecewise constant function, which takes $N$ copies of the sequence of values $\{a, -a, 0, a, -a\}$, i.e.,
\[
\{a, -a, 0, a, -a\ |\ a, -a, 0, a, -a\ |\ \ldots\ |\ a, -a, 0, a, -a\}\,.
\]
After evaluating the state variables in \eqref{theta_arc_int}--\eqref{y_arc_int} at the switching times and carrying out algebraic manipulations, one can equivalently re-write Problem~(Pc) as follows.
\[
\mbox{(Ps)}\left\{\begin{array}{rl}
\min &\ \ds t_N = \sum_{i=1}^N\sum_{j=1}^5 \,\xi^i_j
   \\[4mm]
\mbox{s.t.} &\ds\ x_{i-1} - x_i + \frac{1}{a}\left(-\sin\theta^i_0 + 2\,\sin\theta^i_1 - 2\,\sin\theta^i_2 + 2\,\sin\theta^i_4 - \sin\theta^i_5\right) + \xi^i_3\,\cos\theta^i_2 = 0\,, \\[3mm] 
  &\ds\ y_{i-1} - y_i + \frac{1}{a}\left(\cos\theta^i_0 - 2\,\cos\theta^i_1 + 2\,\cos\theta^i_2 - 2\,\cos\theta^i_4 + \cos\theta^i_5 \right) + \xi^i_3\,\sin\theta^i_2 = 0\,, \\[3mm]
  &\ds\ x_N = x_f\,,\ \ y_N = y_f\,,\ \ \theta^1_0 = \theta_0\,,\ \ \sin\theta^N_5 = \sin\theta_f\,,\ \ \cos\theta^N_5 = \cos\theta_f\,, \\[3mm]
  &\ds\ \xi^i_j \ge 0\,,\quad\mbox{ for } i = 1,\ldots,N\,,\ j = 1,\ldots,5\,, \\[3mm]
  &\ds\ \theta^{i+1}_0 =  \theta^i_5\,,\quad\mbox{ for } i = 1,\ldots,N-1\,,
\end{array}\right.
\]
where
\begin{equation}  \label{thetas_int}
\theta^i_1 = \theta^i_0 + a\,\xi^i_1\,,\qquad
\theta^i_2 = \theta^i_1 - a\,\xi^i_2\,,\qquad
\theta^i_4 = \theta^i_2 + a\,\xi^i_4\,.\qquad
\theta^i_5 = \theta^i_4 - a\,\xi^i_5\,.\qquad
\end{equation}
Substitution of $\theta_1^i$, $\theta_2^i$, $\theta_4^i$ and $\theta_5^i$ in \eqref{thetas_int} into Problem~(Ps) yields a finite dimensional nonlinear optimization problem in $5N$ variables, $\xi^i_j$,  $i = 1,\ldots,N$,  $j = 1,\ldots,5$.

\begin{remark}  \label{sincos} \rm
The constraints
\begin{equation} \label{slope_constr}
\sin\theta_5^N = \sin\theta_f\quad\mbox{and}\quad \cos\theta_5^N = \cos\theta_f\,,
\end{equation}
in Problem~(Ps), ensure that we satisfy the slope condition at the terminal point.  For example, the Markov--Dubins interpolating curve given in Figure~\ref{fig:MD_4pts}(a) for the problem in Example~1 below can be obtained by setting $\theta(t_f) = \theta_f-2\,\pi$, but not by setting $\theta(t_f) = \theta_f$.  The slope condition~\eqref{slope_constr} takes care of such difficulties.
\endproof
\end{remark}

\subsection{Numerical Experiments}
\label{experiments}

In this section, we present numerical experiments by solving Problem~(Ps) under various sets of data to construct Markov--Dubins interpolating curves and other stationary/feasible curves.  For solving Problem~(Ps), we use AMPL~\cite{AMPL} as an optimization modelling language, which employs the optimization software Knitro, version 10.3.0~\cite{Knitro}. In Knitro, the feasibility and optimality tolerances {\tt feastol} and {\tt opttol} were set as $10^{-15}$.  We also allowed Knitro to chose an optimization algorithm appropriately; so, {\tt algo=0} was set.

In the numerical experiments, first the structure of the switchings, i.e. the configuration of the subarcs, is obtained with a coarse tolerance, of say $10^{-8}$.  Once the switching structure is found, the subarcs that are not needed (because their lengths are too small) are excluded from the computations to improve the accuracy of the remaining subarc lengths.

It should be noted that it is also possible to pair up AMPL with other optimization software, e.g., Ipopt~\cite{WacBie2006}, SNOPT\cite{GilMurSau2005} and TANGO~\cite{AndBirMarSch2007,BirMar2014}, in solving Problem~(Ps).

We present numerical experiments in Examples~1--4 below, in which we respectively find 4-, 6-, 20- and 12-point Markov--Dubins interpolating curves.  Define the {\em matrix of subarc lengths} as
\[
\Xi := \left[
\begin{array}{ccccc}
\xi_1^1\ &\ \xi_2^1\ &\ \xi_3^1\ &\ \xi_4^1\ &\ \xi_5^1 \\[2mm]
\xi_1^2\ &\ \xi_2^2\ &\ \xi_3^2\ &\ \xi_4^2\ &\ \xi_5^2 \\
\vdots & \vdots & \vdots & \vdots & \vdots \\[1mm]
\xi_1^N\ &\ \xi_2^N\ &\ \xi_3^N\ &\ \xi_4^N\ &\ \xi_5^N
\end{array} \right],
\]
where the subarc lengths $\xi_j^i$, $i = 1,\ldots,N$, $j = 1,\ldots,5$, are defined as in \eqref{subarc_lengths}.  We list the lengths of the subarcs in each example by means of the matrix $\Xi$.  We also provide the overall length, $t_f$, which  is the sum of all subarc lengths.  The reported numerical results are correct up to 12 decimal places.  We provide the numerical results in much detail so that the instances in Examples~1--4 can be scrutinised and that they may serve as test bed examples in future studies.\\

\afterpage{\clearpage}
\noindent
{\bf Example 1}

\noindent
Consider Problem~(P), or equivalently Problem~(Pc), with the initial and terminal oriented points given as $(x_0,y_0,\theta_0) = (0,0,-\pi/3)$ and $(x_f,y_f,\theta_f) = (1,1,-\pi/6)$.  The intermediate points are taken to be $(x_1,y_1) = (-0.1,0.3)$ and $(x_2,y_2) = (0.2,0.8)$, and the bound on the curvature, $a=3$.  The lower bound on the turning radius is then $1/3$.

In what follows, we list the configurations and the subarc lengths of six feasible solutions of Problem~(P), all of which satisfy the arc-type conditions given in Theorem~\ref{Dubins_int}.  These solutions are graphically depicted in Figure~\ref{fig:MD_4pts}(a)--(f).  The solution curve in Figure~\ref{fig:MD_4pts}(a) has the shortest length and therefore is a Markov--Dubins interpolating curve, at least as far as the intensive numerical experiments conducted indicate.

\ \\
\noindent
(a) $RSL | LSR | RSR \equiv RSLSRSR$\,:\ \ $t_f = 3.415578858075$\,,
{\small
\[
\Xi = \left[
\begin{array}{lllll}
\mbox{\ \ \ \ \ \ \ \ \ 0} & 1.609029653347 & 0.245373087450 & 0.115596919495 & \mbox{\ \ \ \ \ \ \ \ \ 0} \\
0.115596919495 & \mbox{\ \ \ \ \ \ \ \ \ 0} & 0.348770381640 & \mbox{\ \ \ \ \ \ \ \ \ 0} & 0.122237275595 \\
\mbox{\ \ \ \ \ \ \ \ \ \ 0} & 0.122237275595 & 0.439185533812 & \mbox{\ \ \ \ \ \ \ \ \ 0} & 0.297551811646
\end{array} \right].
\]}

\noindent
(b) $RLR | RL | LSR \equiv RLRLSR$\,:\ \ $t_f = 3.859270768865$\,,
{\small
\[
\Xi = \left[
\begin{array}{lllll}
\mbox{\ \ \ \ \ \ \ \ \ 0} & 0.180338361465 & \mbox{\ \ \ \ \ \ \ \ \ 0} & 1.671087869740 & 0.449161039386 \\
\mbox{\ \ \ \ \ \ \ \ \ 0} & 0.660606349458 & \mbox{\ \ \ \ \ \ \ \ \ 0} & 0.040959265073 & \mbox{\ \ \ \ \ \ \ \ \ 0} \\
0.067722881739 & \mbox{\ \ \ \ \ \ \ \ \ 0} & 0.474263660961 & \mbox{\ \ \ \ \ \ \ \ \ 0} & 0.315131341041
\end{array} \right].
\]}

\noindent
(c) $RLR | RL | LR \equiv RLRLR$\,:\ \ $t_f = 4.258605346880$\,,
{\small
\[
\Xi = \left[
\begin{array}{lllll}
\mbox{\ \ \ \ \ \ \ \ \ 0} & 0.014658731348 & \mbox{\ \ \ \ \ \ \ \ \ 0\ \ \ \ \ \ \ \ \ } & 1.660436142087 & 0.198370374835 \\
\mbox{\ \ \ \ \ \ \ \ \ 0} & 1.348732850403 & \mbox{\ \ \ \ \ \ \ \ \ 0} & 0.144567480729 & \mbox{\ \ \ \ \ \ \ \ \ 0} \\
0.411565513223 & 0.480274254254 & \mbox{\ \ \ \ \ \ \ \ \ 0} & \mbox{\ \ \ \ \ \ \ \ \ 0} & \mbox{\ \ \ \ \ \ \ \ \ 0}
\end{array} \right].
\]}

\noindent
(d) $LR | RSL | LSR \equiv LRSLSR$\,:\ \ $t_f = 4.298084620005$\,,
{\small
\[
\Xi = \left[
\begin{array}{lllll}
1.672596123844 & 0.171799627570 & \mbox{\ \ \ \ \ \ \ \ \ 0} & \mbox{\ \ \ \ \ \ \ \ \ 0} & \mbox{\ \ \ \ \ \ \ \ \ 0} \\
\mbox{\ \ \ \ \ \ \ \ \ 0} & 1.364187065025 & 0.033930811053 & 0.191279146930 & \mbox{\ \ \ \ \ \ \ \ \ 0} \\
0.191279146930 & \mbox{\ \ \ \ \ \ \ \ \ 0} & 0.328377898745 & \mbox{\ \ \ \ \ \ \ \ \ 0} & 0.344634799909
\end{array} \right].
\]}

\noindent
(e) $LSL | LR | RSR \equiv LSLRSR$\,:\ \ $t_f = 4.678075540969$\,,
{\small
\[
\Xi = \left[
\begin{array}{lllll}
1.131511003931 &  \mbox{\ \ \ \ \ \ \ \ \ 0} & 0.645570959740 & 1.376095696461 & \mbox{\ \ \ \ \ \ \ \ \ 0} \\
0.475478947958 & 0.161367871190 & \mbox{\ \ \ \ \ \ \ \ \ 0} & \mbox{\ \ \ \ \ \ \ \ \ 0} & \mbox{\ \ \ \ \ \ \ \ \ 0} \\
\mbox{\ \ \ \ \ \ \ \ \ 0} & 0.326240580072 & 0.335261312120 & \mbox{\ \ \ \ \ \ \ \ \ 0} & 0.226549169496
\end{array} \right].
\]}

\noindent
(f) $LRL | LR | RSR \equiv LRLRSR$\,:\ \ $t_f = 4.762973480924$\,,
{\small
\[
\Xi = \left[
\begin{array}{lllll}
1.387975996662 & 0.040303570540 & \mbox{\ \ \ \ \ \ \ \ \ 0} & 0.442471697617 & \mbox{\ \ \ \ \ \ \ \ \ 0} \\
1.532395666196 & 0.398410096474 & \mbox{\ \ \ \ \ \ \ \ \ 0} & \mbox{\ \ \ \ \ \ \ \ \ 0} & \mbox{\ \ \ \ \ \ \ \ \ 0} \\
\mbox{\ \ \ \ \ \ \ \ \ 0} & 0.530782705179 & 0.306214787566 & \mbox{\ \ \ \ \ \ \ \ \ 0} & 0.124418960689
\end{array} \right].
\]}

The remaining solutions in Figure~\ref{fig:MD_4pts}(b)--(f) were reported by Knitro (as well as other optimization software such as Ipopt and SNOPT) as {\em locally optimal}.  However, this does not readily imply that these are locally optimal solutions, or stationary solutions, indeed, of the infinite-dimensional Problem~(P).  We note that the solution in Figure~\ref{fig:MD_4pts}(a) contains only $CSC$ type curves in each stage, and therefore, by Theorem~\ref{stationary_CSC}, the interpolating curve is at least stationary.  To establish stationarity of the other feasible solutions, one needs to check the conditions listed in Theorem~\ref{stationarity_int}.
It is interesting to note that when two stages of type $CSC$ follow one another, as in (a) and (d) above, the lengths of the subarcs just before and just after the relevant node are of equal length, as stated in Proposition~\ref{lengths_CSC}(b) and Remark~\ref{midpoint_CSC}, as a necessary condition of optimality.  This condition is verified in both (a) and (d) above: In (a), $\xi_4^1 = \xi_1^2 = 0.115596919495$ and $\xi_5^2 = \xi_2^3 = 0.122237275595$; in (d), $\xi_4^2 = \xi_1^3 = 0.191279146930$.

Numerical experiments concerning this 4-point interpolation problem yields more than just six solutions---in fact, the number of solutions $M$ is far greater than 10.  As expected, $M$ grows exponentially with the number of interpolant nodes.  As a result, a large number of stationary, or critical, solutions makes it difficult to find a Markov--Dubins interpolating curve, which has the shortest length.

\begin{figure}[t]
\begin{minipage}{80mm}
\begin{center}
\hspace*{-5mm}
\psfrag{x}{\small $x$}
\psfrag{y}{\small $y$}
\includegraphics[width=95mm]{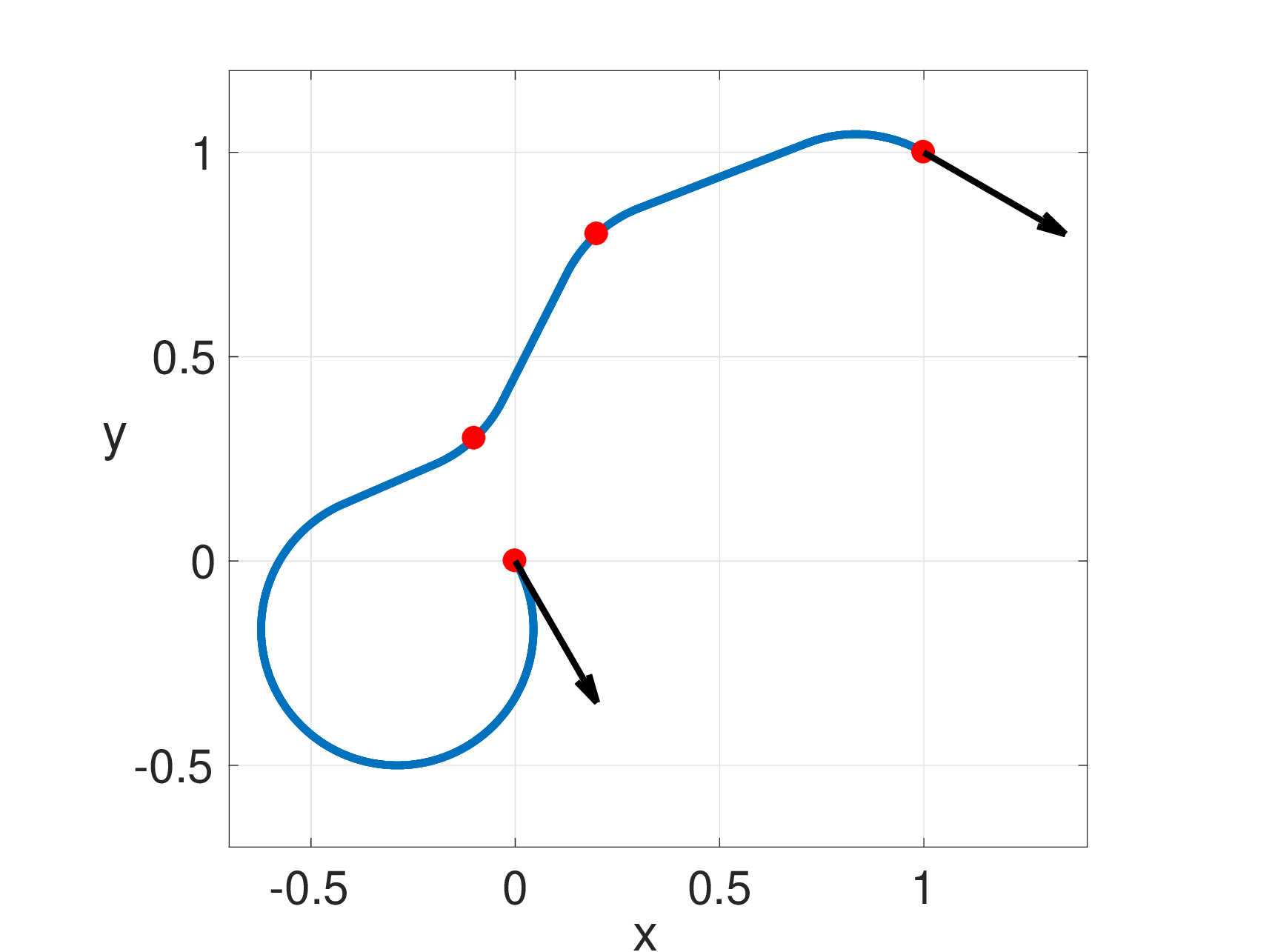} \\[3mm]
(a) {\small Type $RSLSRSR$;\ \ $t_f = 3.415578858075$}
\end{center}
\end{minipage}
\begin{minipage}{80mm}
\begin{center}
\hspace*{-5mm}
\psfrag{x}{\small $x$}
\psfrag{y}{\small $y$}
\includegraphics[width=95mm]{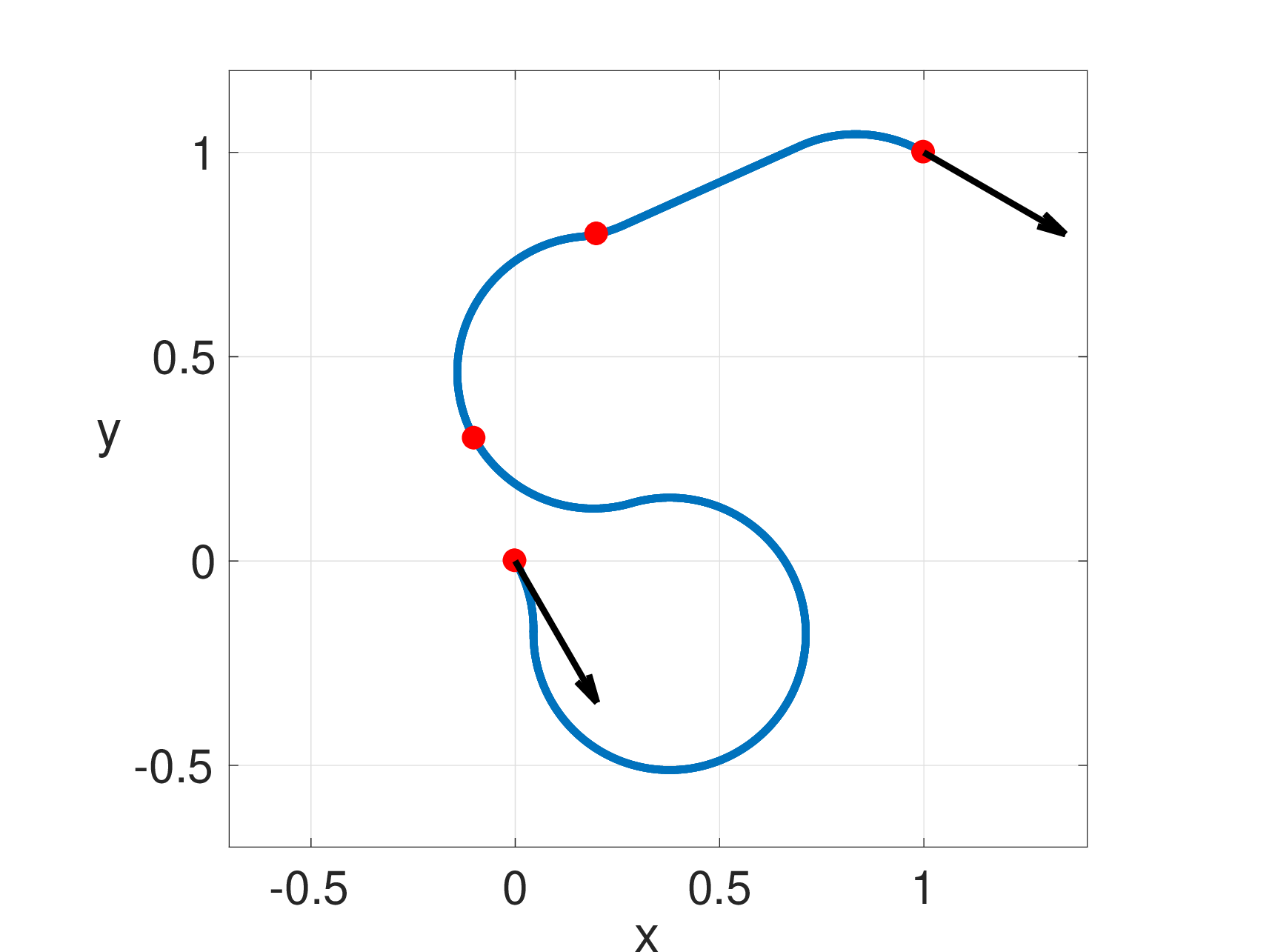} \\[3mm]
(b) {\small Type $RLRLSR$;\ \ $t_f = 3.859270768865$}
\end{center}
\end{minipage}
\begin{minipage}{80mm}
\begin{center}
\hspace*{-5mm}
\psfrag{x}{\small $x$}
\psfrag{y}{\small $y$}
\includegraphics[width=95mm]{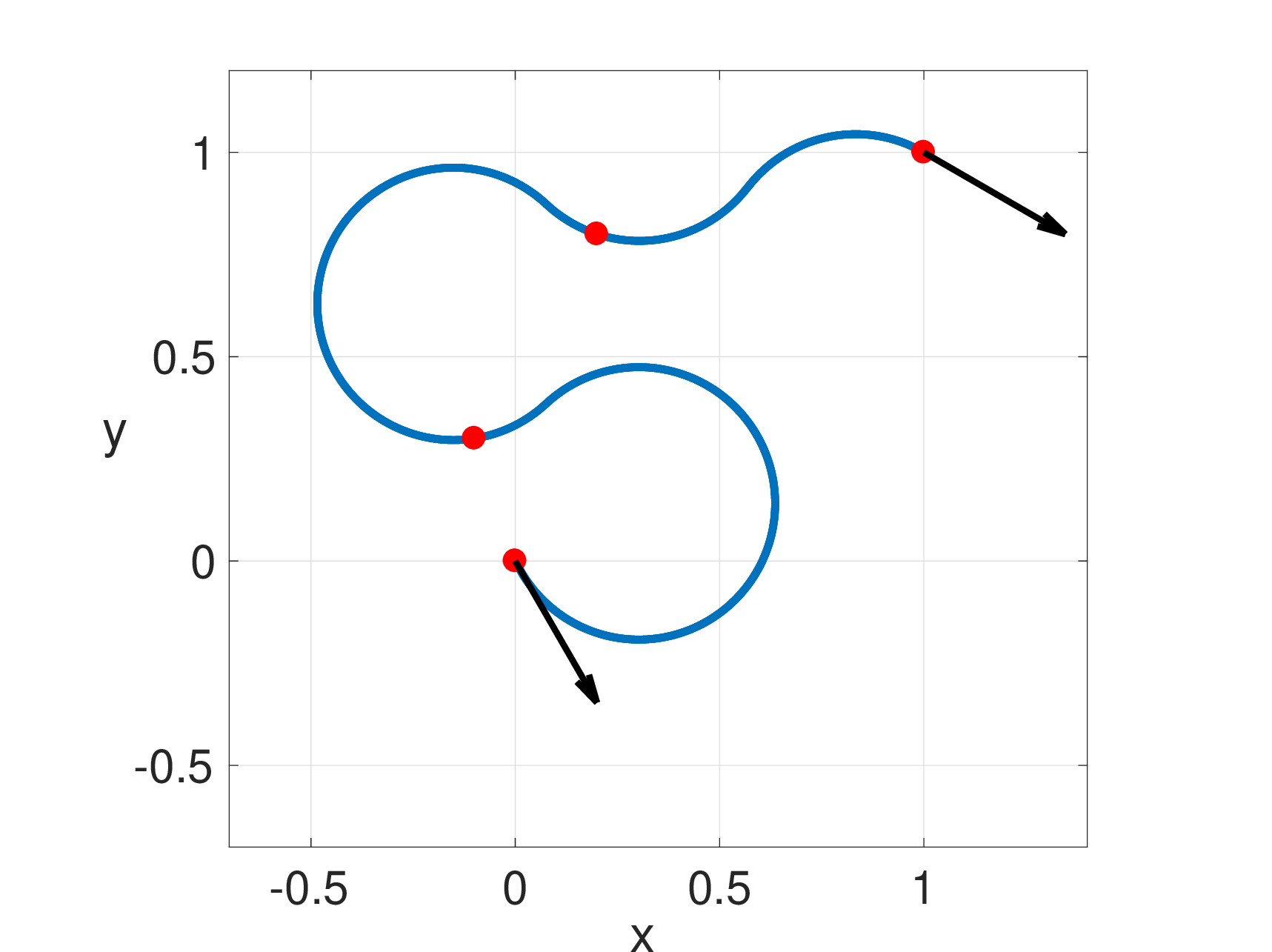} \\[3mm]
(c) {\small Type $RLRLR$;\ \ $t_f = 4.258605346880$}
\end{center}
\end{minipage}
\begin{minipage}{80mm}
\begin{center}
\hspace*{-5mm}
\psfrag{x}{\small $x$}
\psfrag{y}{\small $y$}
\includegraphics[width=95mm]{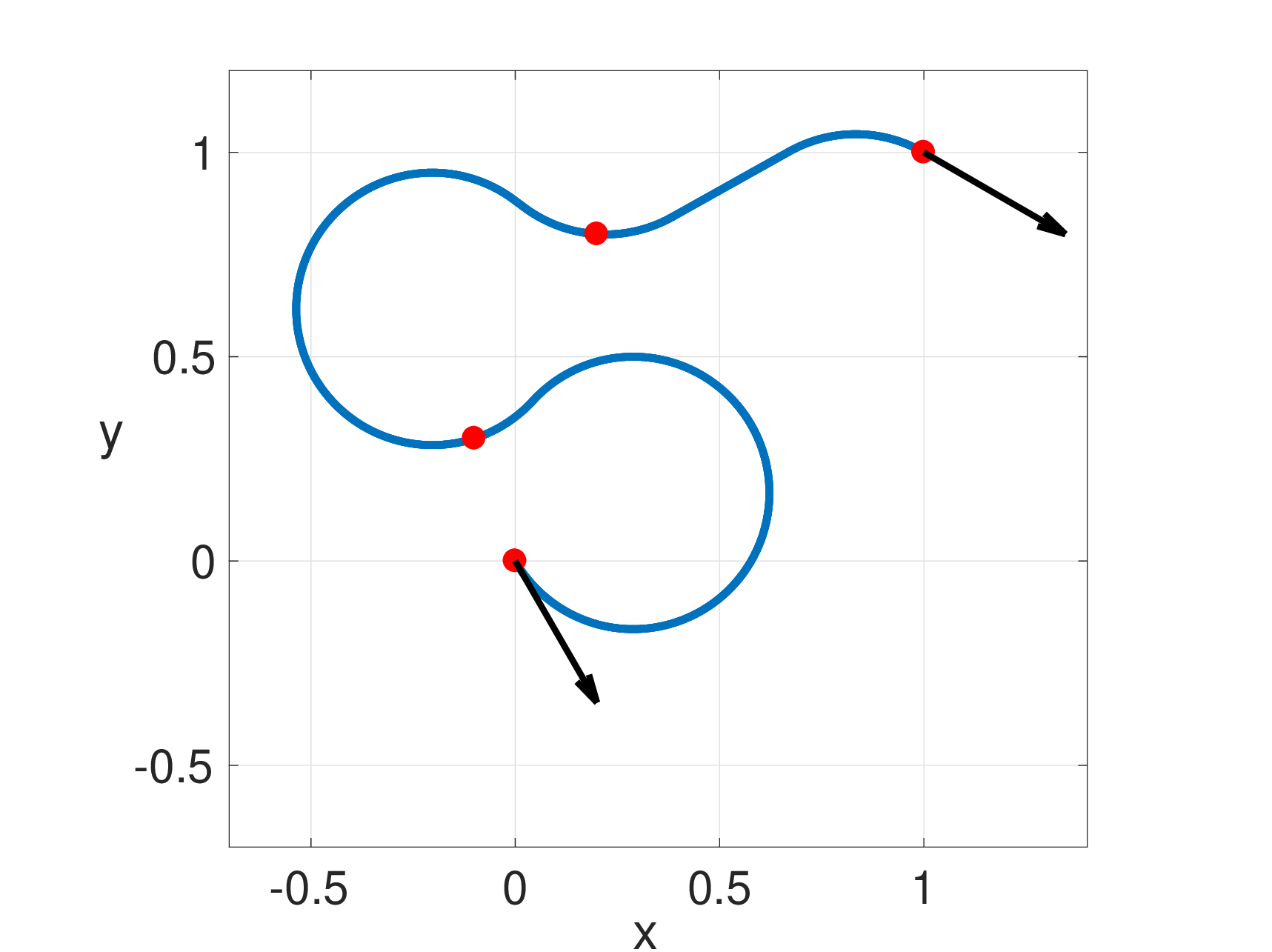} \\[3mm]
(d) {\small Type $LRSLSR$;\ \ $t_f = 4.298084620005$}
\end{center}
\end{minipage}
\begin{minipage}{80mm}
\begin{center}
\hspace*{-5mm}
\psfrag{x}{\small $x$}
\psfrag{y}{\small $y$}
\includegraphics[width=95mm]{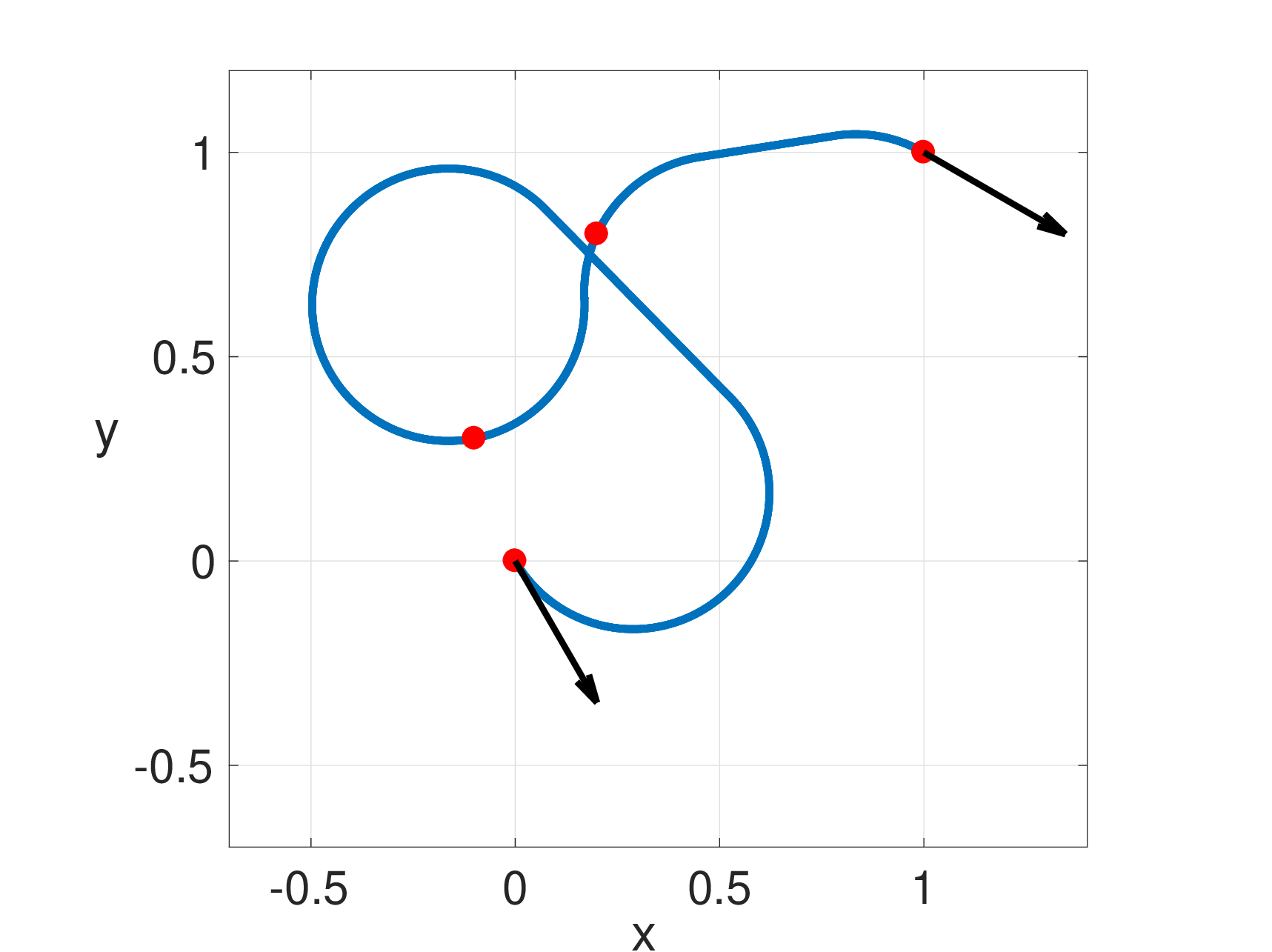} \\[3mm]
(e) {\small Type $LSLRSR$;\ \ $t_f = 4.678075540969$}
\end{center}
\end{minipage}
\begin{minipage}{80mm}
\begin{center}
\psfrag{x}{\small $x$}
\psfrag{y}{\small $y$}
\includegraphics[width=95mm]{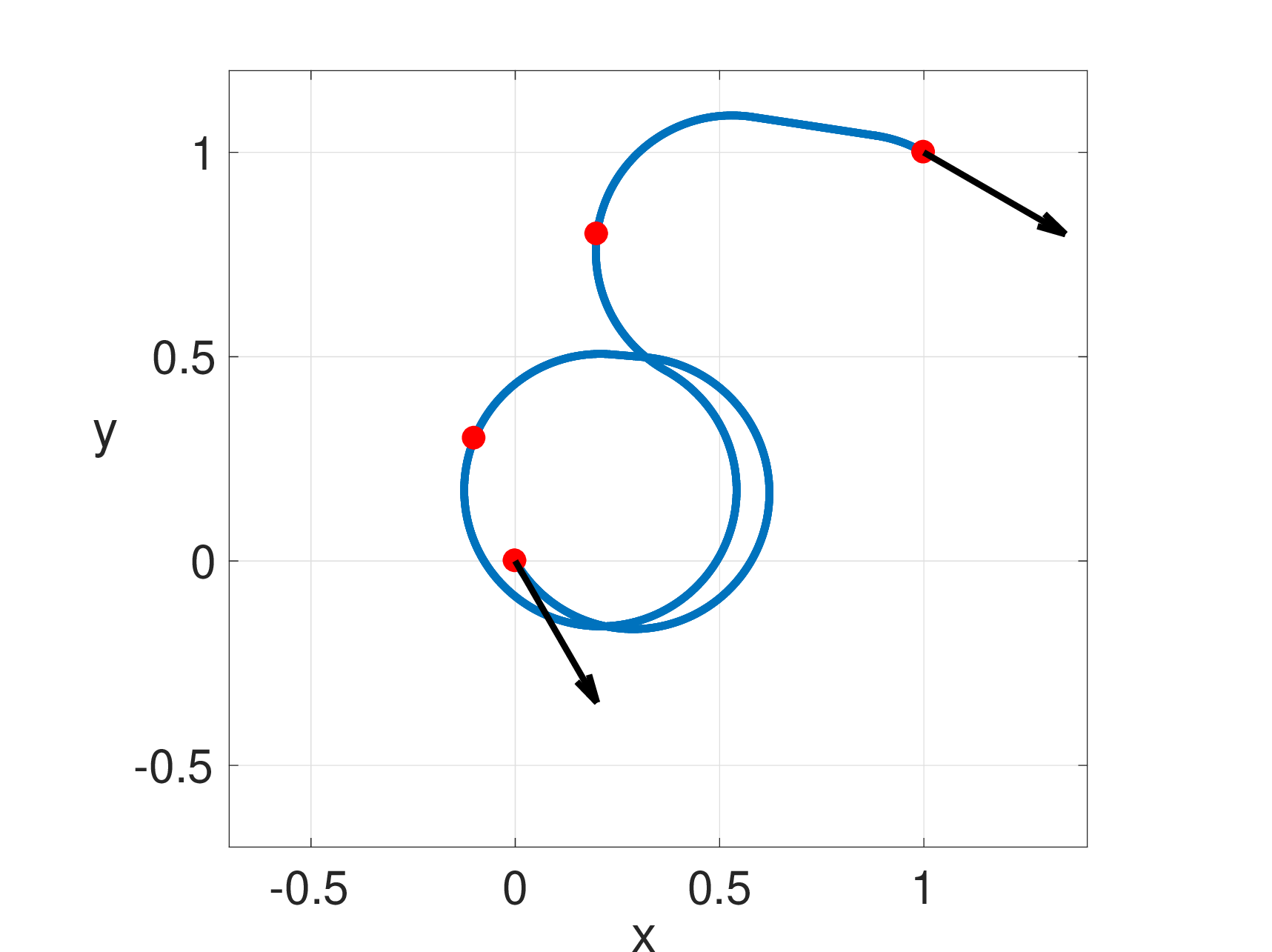} \\[3mm]
(f) {\small Type $LRLRSR$;\ \ $t_f = 4.762973480924$}
\end{center}
\end{minipage}
\
\caption{\small\sf Example~1 -- (a) Markov--Dubins interpolating curve from $(0,0,-\pi/3)$ to $(1,1,-\pi/6)$ via $(-0.1,0.3)$ and $(0.2,0.8)$, with $a=3$, and (b)--(f) some of the other stationary stationary solutions of Problem~(Ps).}
\label{fig:MD_4pts}
\end{figure}

\clearpage

\afterpage{\clearpage}

\noindent
{\bf Example 2}

\noindent
Consider Problem~(P), or equivalently Problem~(Pc).  The initial and terminal oriented points are given as $(x_0,y_0,\theta_0) = (0,0,-\pi/3)$ and $(x_f,y_f,\theta_f) = (0.5,0,-\pi/6)$, respectively.  The intermediate points $(x_i,y_i)$, $i=1,\ldots,4$, are respectively taken to be $(-0.1,0.3)$, $(0.2,0.8)$, $(1,1)$ and $(0.5,0.5)$, and the bound on the curvature is $a=3$.  The lower bound on the turning radius is then $1/3$.

Computational experiments indicate that there are hundreds of  feasible solutions, all of which satisfy the arc-type conditions given in Theorem~\ref{Dubins_int}.  The configurations and the subarc lengths of four (selected) feasible solutions of Problem~(P) are provided below, obtained by solving Problem~(Ps).  These solutions are depicted in Figure~\ref{fig:MD_6pts}(a)--(d).  The solution curve in Figure~\ref{fig:MD_6pts}(a) has the shortest length we were able to find, and so it is declared here to be a Markov--Dubins interpolating curve, at least as far as the computations carried out in this paper are concerned.

\ \\
\noindent
(a) $RSL | LSR | RSR | RSR | RLR \equiv RSLSRSRSRLR$\,:\ \ $t_f = 6.278034550309$\,,
{\small
\[
\Xi = \left[
\begin{array}{lllll}
\mbox{\ \ \ \ \ \ \ \ \ 0} & 1.607146208885 & 0.253152303916 & 0.109461129478 & \mbox{\ \ \ \ \ \ \ \ \ 0} \\
0.109461129478 & \mbox{\ \ \ \ \ \ \ \ \ 0} & 0.411866814272 & \mbox{\ \ \ \ \ \ \ \ \ 0} & 0.063620967753 \\
\mbox{\ \ \ \ \ \ \ \ \ 0} & 0.063620967753 & 0.349008605883 & \mbox{\ \ \ \ \ \ \ \ \ 0} & 0.551024831028 \\
\mbox{\ \ \ \ \ \ \ \ \ 0} & 0.551024831028 & 0.055775140041 & \mbox{\ \ \ \ \ \ \ \ \ 0} & 0.362796821592 \\
\mbox{\ \ \ \ \ \ \ \ \ 0} & 0.105078700947 & \mbox{\ \ \ \ \ \ \ \ \ 0} & 1.425262495545 & 0.259733602711
\end{array} \right].
\]}

\noindent
(b) $RSL | LSR | RSR | RLR | LR \equiv RSLSRSRLRLR$\,:\ \ $t_f = 6.488873243877$\,,
{\small
\[
\Xi = \left[
\begin{array}{lllll}
\mbox{\ \ \ \ \ \ \ \ \ 0} & 1.608655551819 & 0.246889937788 & 0.114406041268 & \mbox{\ \ \ \ \ \ \ \ \ 0} \\
0.114406041268 & \mbox{\ \ \ \ \ \ 0} & 0.358542879421 & \mbox{\ \ \ \ \ \ \ \ \ 0} & 0.113274189452 \\
\mbox{\ \ \ \ \ \ \ \ \ 0} & 0.113274189452 & 0.416609605051 & \mbox{\ \ \ \ \ \ \ \ \ 0} & 0.341389286409 \\
\mbox{\ \ \ \ \ \ \ \ \ 0} & 0.364397523143 & \mbox{\ \ \ \ \ \ \ \ \ 0} & 0.045089419939 & 0.908572347990 \\
1.499582819736 & 0.243783411139 & \mbox{\ \ \ \ \ \ \ \ \ 0} & \mbox{\ \ \ \ \ \ \ \ \ 0} & \mbox{\ \ \ \ \ \ \ \ \ 0}
\end{array} \right].
\]}

\noindent
(c) $RLR | RL | LSR | RSR | RLR \equiv RLRLSRSRLR$\,:\ \ $t_f = 6.729555454357$\,,
{\small
\[
\Xi = \left[
\begin{array}{lllll}
\mbox{\ \ \ \ \ \ 0} & 0.185101731608 & \mbox{\ \ \ \ \ \ 0} & 1.673440788217 & 0.456049670642 \\
\mbox{\ \ \ \ \ \ 0} & 0.622953488994 & \mbox{\ \ \ \ \ \ 0} & 0.066834089291 & \mbox{\ \ \ \ \ \ 0} \\
0.121935127737 & \mbox{\ \ \ \ \ \ 0} & 0.272852512220 & \mbox{\ \ \ \ \ \ 0} & 0.565210311199 \\
\mbox{\ \ \ \ \ \ 0} & 0.565210311199 & 0.054452710847 & \mbox{\ \ \ \ \ \ 0} & 0.357505055062 \\
\mbox{\ \ \ \ \ \ 0} & 0.102754639709 & \mbox{\ \ \ \ \ \ 0} & 1.426181573000 & 0.259073444632
\end{array} \right].
\]}

\noindent
(d) $RLR | RL | LSR | RSR | LR \equiv RLRLSRSRLR$\,:\ \ $t_f = 6.933659387154$\,,
{\small
\[
\Xi = \left[
\begin{array}{lllll}
\mbox{\ \ \ \ \ \ 0} & 0.180848776168 & \mbox{\ \ \ \ \ \ 0} & 1.671336709782 & 0.449898895805 \\
\mbox{\ \ \ \ \ \ 0} & 0.655840953047 & \mbox{\ \ \ \ \ \ 0} & 0.044391479634 & \mbox{\ \ \ \ \ \ 0} \\
0.074776055868 & \mbox{\ \ \ \ \ \ 0} & 0.439394036526 & \mbox{\ \ \ \ \ \ 0} & 0.354344875528 \\
\mbox{\ \ \ \ \ \ 0} & 0.354344875528 & 0.088624145788 & \mbox{\ \ \ \ \ \ 0} & 0.876492352605 \\
1.499582819736 & 0.243783411139 & \mbox{\ \ \ \ \ \ 0} & \mbox{\ \ \ \ \ \ 0} & \mbox{\ \ \ \ \ \ 0}
\end{array} \right].
\]}

\begin{figure}[t]
\begin{minipage}{80mm}
\begin{center}
\hspace*{-5mm}
\psfrag{x}{\small $x$}
\psfrag{y}{\small $y$}
\includegraphics[width=95mm]{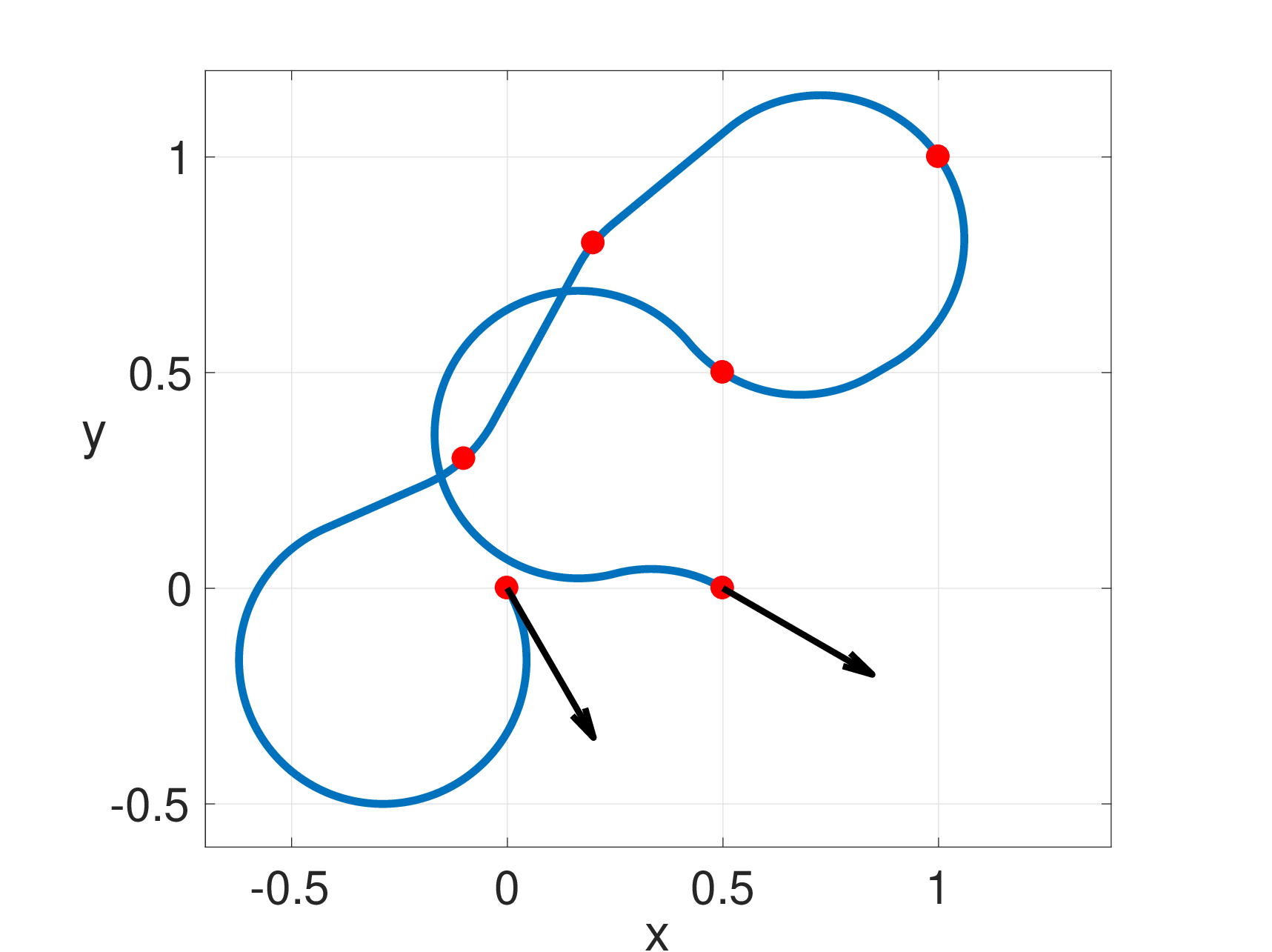} \\[3mm]
(a) {\small Type $RSLSRSRSRLR$;\ \ $t_f = 6.2780346$}
\end{center}
\end{minipage}
\begin{minipage}{80mm}
\begin{center}
\hspace*{-5mm}
\psfrag{x}{\small $x$}
\psfrag{y}{\small $y$}
\includegraphics[width=95mm]{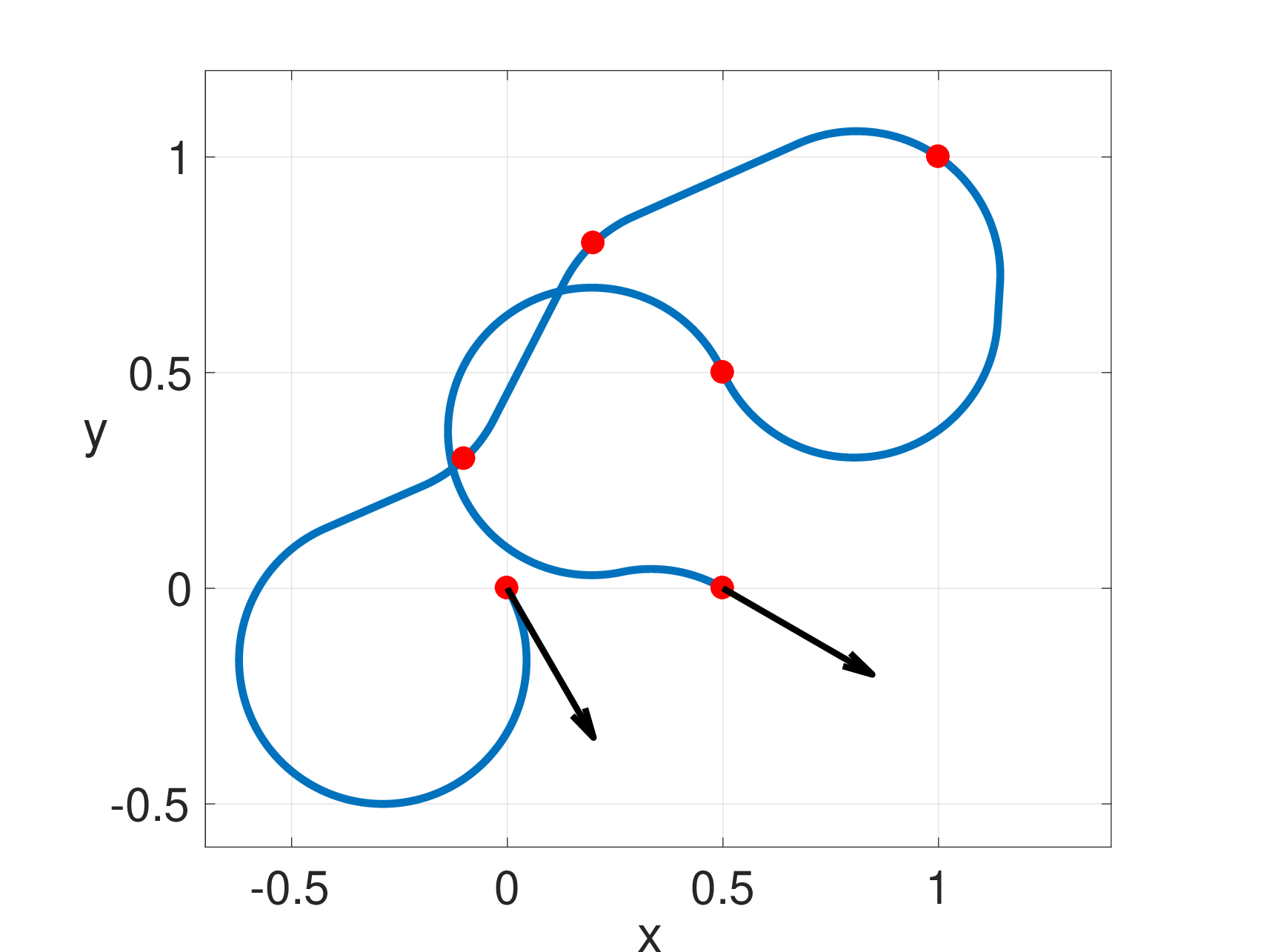} \\[3mm]
(b) {\small Type $RSLSRSRLRLR$;\ \ $t_f = 6.4888732$}
\end{center}
\end{minipage}
\begin{minipage}{80mm}
\begin{center}
\hspace*{-5mm}
\psfrag{x}{\small $x$}
\psfrag{y}{\small $y$}
\includegraphics[width=95mm]{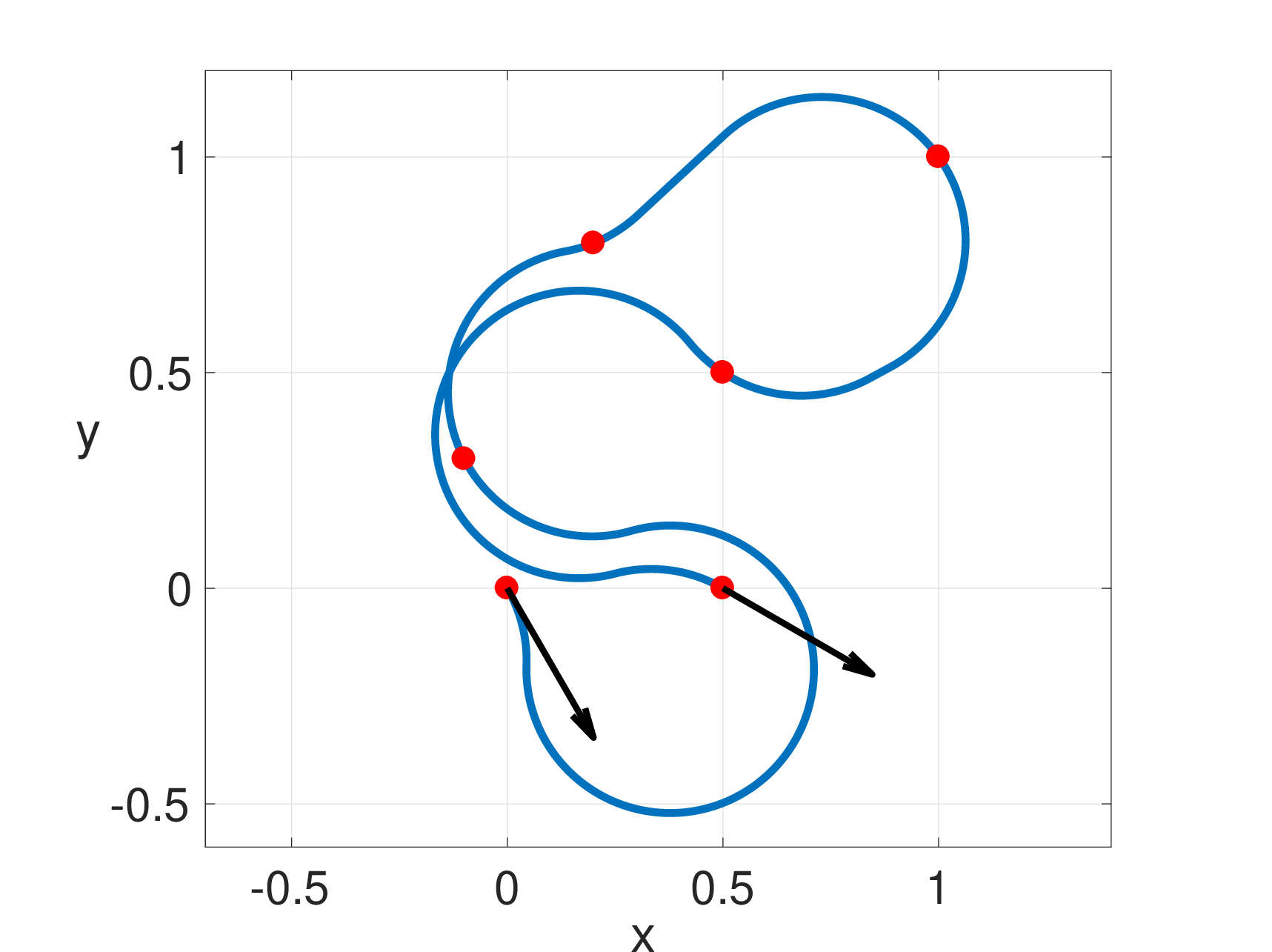} \\[3mm]
(c) {\small Type $RLRLSRSRLR$;\ \ $t_f = 6.7295555$}
\end{center}
\end{minipage}
\begin{minipage}{80mm}
\begin{center}
\hspace*{-1mm}
\psfrag{x}{\small $x$}
\psfrag{y}{\small $y$}
\includegraphics[width=95mm]{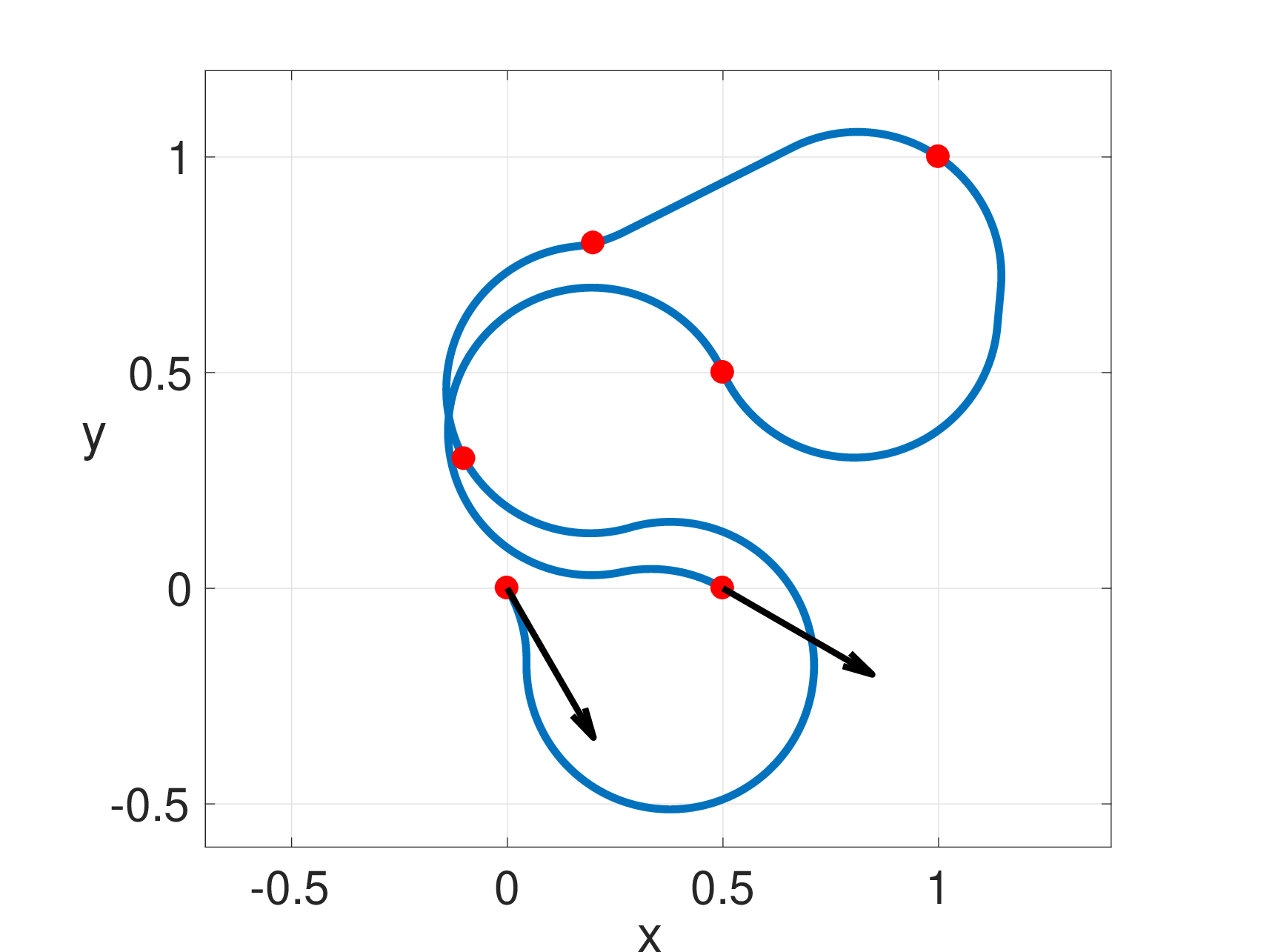} \\[3mm]
(d) {\small Type $RLRLSRSRLR$;\ \ $t_f = 6.9336594$}
\end{center}
\end{minipage}
\
\caption{\small\sf Example 2 -- (a) Markov--Dubins interpolating curve from $(0,0,-\pi/3)$ to $(0.5,0,-\pi/6)$ via $(-0.1,0.3)$, $(0.2,0.8)$, $(1,1)$ and $(0.5,0.5)$, with $a=3$, and (b)--(d) some of the other feasible solutions of Problem~(Ps).}
\label{fig:MD_6pts}
\end{figure}

For the stationarity of the feasible solutions in (b)--(d) above, the conditions listed in Theorem~\ref{stationarity_int} can be checked. It should be noted that all of the solutions listed in (a)-(d) above have CSC solutions in some two consecutive stages, and that they verify the condition stated in Proposition~\ref{lengths_CSC}(b) and Remark~\ref{midpoint_CSC}, as a necessary condition of optimality: In (a), $\xi_4^1 = \xi_1^2$, $\xi_5^2 = \xi_2^3$ and $\xi_5^3 = \xi_2^4$; in (b), $\xi_4^1 = \xi_1^2$, $\xi_5^2 = \xi_2^3$; and in (c) and (d), $\xi_5^3 = \xi_2^4$.

It is interesting to note that, in (b) and (d), a switching from an $R$-subarc to an $L$-subarc occurs exactly at the last interior node $(x_5,y_5) = (0.5,0.5)$. Switchings at the nodes are rare events, as pointed earlier.  However, these events happen here conceivably because of the particular choices of the interior nodes as well as the maximum allowed curvature $a$.

\clearpage

\noindent
{\bf Example 3}

\noindent
In this example, we consider 20 points.  The initial and terminal oriented points are taken to be $(x_0,y_0,\theta_0) = (0.5,1.2,,5\pi/6)$ and $(x_f,y_f,\theta_f) = (2.5,0.6,0)$, respectively.  The intermediate points $(x_i,y_i)$, $i=1,\ldots,18$, respectively are
{\small
\[
\left[
\begin{array}{ccc}
p_1 & \cdots & p_9 \\[1mm]
p_{10} & \cdots & p_{18}
\end{array} \right] =
\left[
\begin{array}{ccccccccc}
(0, 0.8) & (0, 0.4) & (0.1,0) & (0.4, 0.2) & (0.5, 0.5) & (0.6, 1) & (1, 0.8) & (1, 0) & (1.4, 0.2) \\[1mm]
(1.2, 1) & (1.5, 1.2) & (2, 1.5) & (1.5, 0.8) & (1.5, 0) & (1.7, 0.6) & (1.9, 1) & (2, 0.5) & (1.9, 0)
\end{array} \right].
\]}
The bound on the curvature is $a = 5$; namely the minimum turning radius is $0.2$.  These kinds of interpolation problems arise in land and marine surveillance, including military and civilian search-and-rescue operations.  The configuration and subarc lengths of the shortest length solution of Problem~(Ps) we were able to find using AMPL and Knitro are provided below.
{\small
\[
LSL | LSR | RSL | LSL | LSL | LSR | RSR | RSL | LSL | LSR | RSL | LSR | RSL | LSL | LSL | LSR | RSR | RSL | LSR
\]
\[
\equiv LSLSRSLSLSLSRSRSLSLSRSLSRSLSLSLSRSRSLSR\,:\ \ t_f = 11.916212654286\,,
\]}
{\small
\[
\Xi = \left[
\begin{array}{lllll}
0.292683660485 & \mbox{\ \ \ \ \ \ \ \ \ 0} & 0.354227249883 & 0.066067208642 & \mbox{\ \ \ \ \ \ \ \ \ 0} \\
0.066067208642 & \mbox{\ \ \ \ \ \ \ \ \ 0} & 0.314358037636 & \mbox{\ \ \ \ \ \ \ \ \ 0} & 0.020629993182 \\
\mbox{\ \ \ \ \ \ \ \ \ 0} & 0.020629993182 & 0.158248660263 & 0.281673366237 & \mbox{\ \ \ \ \ \ \ \ \ 0} \\
0.281673366237 & \mbox{\ \ \ \ \ \ \ \ \ 0} & 0.105094394904 & 0.017147975416 & \mbox{\ \ \ \ \ \ \ \ \ 0}  \\
0.017147975416 & \mbox{\ \ \ \ \ \ \ \ \ 0} & 0.262701065614 & 0.036592830635 & \mbox{\ \ \ \ \ \ \ \ \ 0}  \\
0.036592830635 & \mbox{\ \ \ \ \ \ \ \ \ 0}  & 0.278860332397 & \mbox{\ \ \ \ \ \ \ \ \ 0} & 0.225886597060 \\
\mbox{\ \ \ \ \ \ \ \ \ 0} & 0.225886597060 & 0.151864534206 & \mbox{\ \ \ \ \ \ \ \ \ 0} & 0.112422725874 \\
\mbox{\ \ \ \ \ \ \ \ \ 0} & 0.112422725874 & 0.488292307990 & 0.245323671189 & \mbox{\ \ \ \ \ \ \ \ \ 0}  \\
0.245323671189 & \mbox{\ \ \ \ \ \ \ \ \ 0} & 0.131702140115 & 0.125114048917 & \mbox{\ \ \ \ \ \ \ \ \ 0}  \\
0.125114048917 & \mbox{\ \ \ \ \ \ \ \ \ 0} & 0.565103190136 & \mbox{\ \ \ \ \ \ \ \ \ 0} & 0.151217907866 \\
\mbox{\ \ \ \ \ \ \ \ \ 0} & 0.151217907866 & 0.164572410900 & 0.054093591072 & \mbox{\ \ \ \ \ \ \ \ \ 0}  \\
0.054093591072 & \mbox{\ \ \ \ \ \ \ \ \ 0} & 0.281891534948 & \mbox{\ \ \ \ \ \ \ \ \ 0} & 0.342811045871 \\
\mbox{\ \ \ \ \ \ \ \ \ 0} & 0.342811045871 & 0.568086259614 & 0.061595474597 & \mbox{\ \ \ \ \ \ \ \ \ 0}  \\
0.061595474597 & \mbox{\ \ \ \ \ \ \ \ \ 0} & 0.510111126314 & 0.348337479570 & \mbox{\ \ \ \ \ \ \ \ \ 0}  \\
0.348337479569 & \mbox{\ \ \ \ \ \ \ \ \ 0} & 0.386735718514 & 0.003761723139 & \mbox{\ \ \ \ \ \ \ \ \ 0}  \\
0.003761723139 & \mbox{\ \ \ \ \ \ \ \ \ 0} & 0.178946350711 & \mbox{\ \ \ \ \ \ \ \ \ 0} & 0.351310541702 \\
\mbox{\ \ \ \ \ \ \ \ \ 0} & 0.351310541702 & 0.216552941908 & \mbox{\ \ \ \ \ \ \ \ \ 0} & 0.040835697925 \\
\mbox{\ \ \ \ \ \ \ \ \ 0} & 0.040835697925 & 0.212956892491 & 0.349098604305 & \mbox{\ \ \ \ \ \ \ \ \ 0}  \\
0.349098604305 & \mbox{\ \ \ \ \ \ \ \ \ 0} & 0.378354575733 & \mbox{\ \ \ \ \ \ \ \ \ 0} & 0.247028303126
\end{array} \right].
\]}

Figure~\ref{fig:MD_20pts} depicts the computed interpolating curve.  it should be noted that $\xi_3^i\neq0$, $i=1,\ldots,19$, and that each stage arc is of type $CSC$.  So, by Theorem~\ref{stationary_CSC}, the solution is proved to be at least stationary.  It can be easily checked by inspection that the tabulated subarc lengths in matrix $\Xi$ above also verify the necessary conditions of optimality in Proposition~\ref{lengths_CSC} and Remark~\ref{midpoint_CSC}, in that each node subdivides a $C$-type subarc into two $C$-type subarcs of equal lengths and that the length of each $C$-type subarc is less than $\pi$.  To verify that the stationary interpolating curve found in this example is locally optimal, further analysis is needed, e.g., by using the theory and computational approaches for second-order sufficient conditions of optimality given in \cite{MauBueKimKay2005, OsmMau2012,AroBonDmiLot2012}.

There seem to be thousands of solutions of (Ps) which satisfy the arc-type conditions given in Theorem~\ref{Dubins_int}.  Therefore, finding a global optimal solution is a much greater challenge for this particular interpolation problem which has a large number of nodes.

\begin{figure}[t]
\[
\hspace*{-5mm}
\psfrag{x}{$x$}
\psfrag{y}{$y$}
\includegraphics[width=130mm]{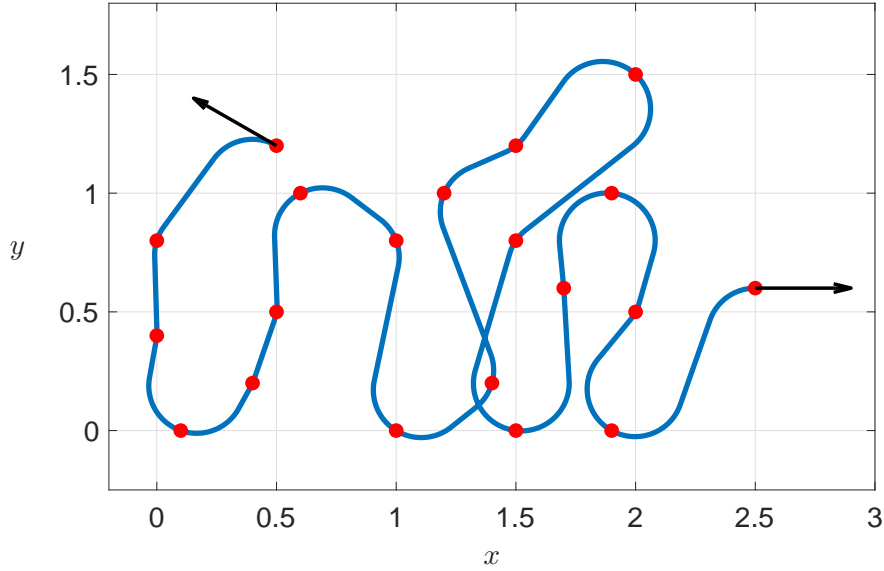}
\]
\caption{\small\sf Example~3 -- A 20-node Markov--Dubins interpolating curve, whose length is $t_f = 11.916212654286$, from $(0.5,1.2,5\pi/6)$ to $(2.5,0.6,0)$, with $a=3$.}
\label{fig:MD_20pts}
\end{figure}

\newpage

\noindent
{\bf Example 4}

\noindent
In this example, we also consider a large number of nodes for the interpolating curve to pass through, but with a ``shape'' or ``pattern."  The initial and terminal oriented points are given as $(x_0,y_0,\theta_0) = (0.5,1.2,,5\pi/6)$ and $(x_f,y_f,\theta_f) = (0,-0.5,0)$, respectively.  The intermediate points $(x_i,y_i)$, $i=1,\ldots,10$, respectively are
\[
\left[
\begin{array}{ccc}
p_1 & \cdots & p_5 \\[1mm]
p_6 & \cdots & p_9
\end{array} \right] =
\left[
\begin{array}{ccccc}
(0.0, 0.5) & (0.5, 0.5) & (1.0, 0.5) & (1.5, 0.5) & (2.0, 0.5) \\[1mm]
(2.0, 0.0) & (1.5, 0.0) & (1.0, 0.0) & (0.5, 0.0) & (0.0, 0.0)
\end{array} \right].
\]
The bound on the curvature is $a = 3$; namely the minimum turning radius is $1/3$.  These interpolation problems  arise in applications where it is necessary to map an area completely for the purposes of military or civilian surveillance, agriculture, etc.  The configuration and subarc lengths of the shortest length solution of Problem~(Ps) we were able to find using AMPL and Knitro are provided below.

\noindent
{\small$LSL | LR | RSL | LSL | LSR | R | RSL | LSR | RSL | LSR | RLR \equiv LSLRSLSLSRSLSRSLSRLR$\,: \\[1mm]
\hspace*{10mm}$t_f = 7.467562181965$\,,}

{\small
\[
\Xi = \left[
\begin{array}{lllll}
0.517980939547 & \mbox{\ \ \ \ \ \ \ \ \ 0} & 0.199236689725 & 0.448783310430 & \mbox{\ \ \ \ \ \ \ \ \ 0} \\
0.444952611925 & 0.098280826419 & \mbox{\ \ \ \ \ \ \ \ \ 0} & \mbox{\ \ \ \ \ \ \ \ \ 0} & \mbox{\ \ \ \ \ \ \ \ \ 0} \\
\mbox{\ \ \ \ \ \ \ \ \ 0} & 0.102046764427 & 0.396637972184 & 0.002661244193 & \mbox{\ \ \ \ \ \ \ \ \ 0} \\
0.002661244193 & \mbox{\ \ \ \ \ \ \ \ \ 0} & 0.426792526518 & 0.071025820394 & \mbox{\ \ \ \ \ \ \ \ \ 0} \\
0.071025820394 & \mbox{\ \ \ \ \ \ \ \ \ 0} & 0.085837912032 & \mbox{\ \ \ \ \ \ \ \ \ 0} & 0.377944339773 \\
\mbox{\ \ \ \ \ \ \ \ \ 0} & 0.565374719321  & \mbox{\ \ \ \ \ \ \ \ \ 0} & \mbox{\ \ \ \ \ \ \ \ \ 0} & \mbox{\ \ \ \ \ \ \ \ \ 0} \\
\mbox{\ \ \ \ \ \ \ \ \ 0} & 0.377949440124 & 0.085821641969 & 0.071037130565 & \mbox{\ \ \ \ \ \ \ \ \ 0} \\
0.071037130565 & \mbox{\ \ \ \ \ \ \ \ \ 0} & 0.426468907257 & \mbox{\ \ \ \ \ \ \ \ \ 0} & 0.002973423917 \\
\mbox{\ \ \ \ \ \ \ \ \ 0} & 0.002973423917 & 0.467025768202 & 0.030039625776 & \mbox{\ \ \ \ \ \ \ \ \ 0} \\
0.030039625776 & \mbox{\ \ \ \ \ \ \ \ \ 0} & 0.237647474938 & \mbox{\ \ \ \ \ \ \ \ \ 0} & 0.246307287638 \\
\mbox{\ \ \ \ \ \ \ \ \ 0} & 0.052184608613 & \mbox{\ \ \ \ \ \ \ \ \ 0} & 1.420667379008 & 0.134146572224
\end{array} \right].
\]}

Figure~\ref{fig:MD_12pts} depicts the computed interpolating curve. For the stationarity of this solution, the conditions listed in Theorem~\ref{stationarity_int} can be checked. It should be noted the stages 2--5 and 6--9 have CSC solutions, and so the necessary condition of optimality stated in Proposition~\ref{lengths_CSC}(b) and Remark~\ref{midpoint_CSC}, is verified for any two consecutive stage curves of type $CSC$.

\begin{figure}[t]
\[
\hspace*{-5mm}
\psfrag{x}{$x$}
\psfrag{y}{$y$}
\includegraphics[width=120mm]{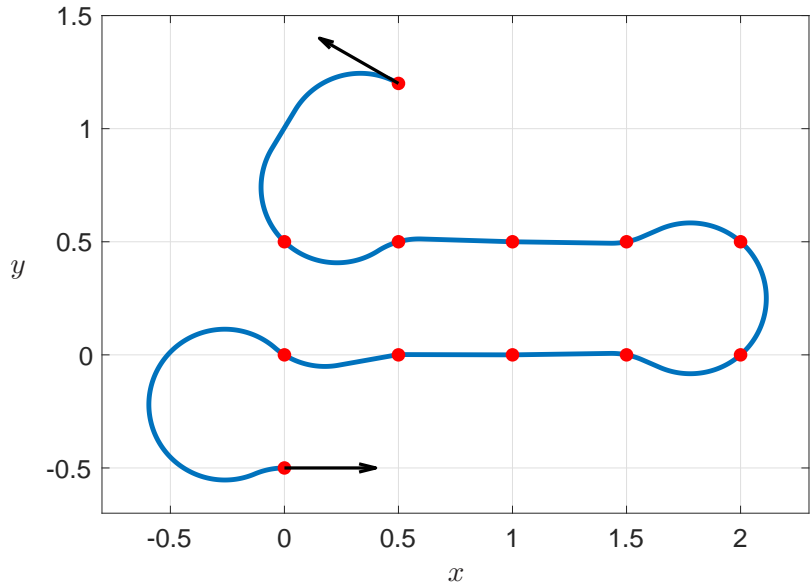}
\]
\caption{\small\sf Example~4 -- A 12-node Markov--Dubins interpolating curve, whose length is $t_f = 7.467562181965$, from $(0.5,1.2,5\pi/6)$ to $(0,-0.5,0)$, with $a=3$.}
\label{fig:MD_12pts}
\end{figure}

\section{Conclusion and Open Problems}
\label{conclusion}

We have studied the Markov--Dubins interpolation problem by employing optimal control theory.  The problem is a natural extension of the Markov--Dubins problem, in that it requires the shortest curve of constrained curvature between two oriented points to pass through a number of intermediate points.  We have shown that an optimal control formulation of  the interpolation problem can be abnormal, as well as normal (Lemmas~\ref{rho_nonsing_int} and \ref{abnormal_int}).   We characterized the associated solutions: in a stage, i.e. between any two consecutive nodes, one has a curve of type $C$ or $CC$ when the problem is abnormal, or type $CCC$ or $CSC$ or a subset thereof when the problem is normal (Theorem~\ref{Dubins_int}).  We established that if the problem is abnormal, then the whole interpolating curve is abnormal, i.e., of type $C\cdots C$. 

Under the assumption that the curvature at the nodes are continuous, in a curve with $(N+1)$ nodes, the number of subarcs is bounded above by $(2N+1)$ (Proposition~\ref{subarcs_int}).  For stationarity, or criticality, of the interpolating curve, certain junction conditions (at the nodes) need to be satisfied (Theorems~\ref{stationarity_int} and \ref{stationarity_int_abnormal}).  These junction conditions are satisfied automatically and so the interpolating curve is indeed stationary, if the stage curves are of type $CSC$ (Theorem~\ref{stationary_CSC}), i.e. the interpolating curve is of type $CSCSC\cdots CSC$, which is the typical structure of the optimal solution when each two consecutive nodes are far enough from one another.  

We have also obtained a necessary condition of optimality: When two consecutive stage arcs of the optimal interpolating curve are of type $CSC$ the node between the two stages subdivide the $C$ subarc into two segments of equal length (Proposition~\ref{lengths_CSC}).  This result has already been useful in eliminating some of the computational solutions (which did not satisfy the condition).

We have proposed a numerical method in finding Markov--Dubins interpolating curves by utilizing arc-parameterization techniques from earlier work.  The approach we had developed in \cite{Kaya2017} for Markov--Dubins path has served as a building block, as was promised in that paper.  We presented examples with small as well as relatively large number of nodes, which both illustrated the numerical method and verified the theoretical results.  

Reference \cite{Kaya2017} already hints via examples that one should expect to find many stationary solutions of the Markov--Dubins problem, i.e., Problem~(P) with $N=1$, by using finite-dimensional optimization methods for solving the induced arc-parameterized problem~(Ps).  Optimization software associated with these methods declare that a solution that it finds is ``locally optimal'' for the arc-parameterized problem.  On the other hand, local optimality, or even stationarity, may not in general be true for the infinite dimensional Problem~(P).  Finding a globally optimal solution can therefore be a real challenge for Problem~(P) with $N>1$, since the number of stationary solutions found by the optimization software for the generalized arc-parameterized problem is expected to grow exponentially with the number nodes, $(N+1)$.

Establishing the nature of the stationaty solutions of the Markov--Dubins interpolating problem remains an open problem as in the case of $N=1$.  For establishing second-order sufficient conditions of optimality, the results and numerical implementation in references~\cite{MauBueKimKay2005, OsmMau2012, AroBonDmiLot2012} could perhaps be utilized, and this is a topic of future research.

In practical situations, for example in the optimal flight trajectory planning of a drone or in the path optimization of underground mine tunnelling, constraints are often imposed because of the features of the terrain and the no-go areas.  To the author's knowledge, Micchelli, Smith, Swetits and Ward were the first to tackle certain scalar interpolation problems with spatial constraints in \cite{MicSmiSweWar1985}, where they required convexity of the interpolants.  It is well-known that an interpolating curve minimizing the $L^2$-norm of its acceleration is a piecewise cubic spline.  References~\cite{Dontchev1993, DonKol1996, DonQiQiYin2002} investigated such problems with various types of constraints such as ``strips'' between consecutive data points and convexity. The works~\cite{FreObeOpf1999, OpfObe1988} studied restricted as well as monotone cubic spline interpolants by treating the interpolation problem as an optimal control problem.  These earlier efforts justify the consideration of spatial constraints in the future for the Markov--Dubins interpolating problem, with $N\ge2$, which is well-known to be more challenging.

In certain situations, it would be of interest to provide the intermediate nodes in no particular order, i.e., the order in which the intermediate nodes are ``visited'' would also need to be optimized.  In the case when the curve ends where it started, this kind of problem can be modelled as a {\em travelling salesperson problem}, as the way it is considered in~\cite{IsaShi2015, SavFraBul2008}.  In both~\cite{IsaShi2015, SavFraBul2008}, the problems are studied by heuristic approaches. This is by every means very valuable in practical applications; however, it would be interesting to study these kinds of problems by the setting introduced in this paper.

\end{document}